\documentclass[11pt]{amsart}
\usepackage[a4paper, margin=3cm]{geometry}
\usepackage{microtype}
\usepackage[english]{babel}
\usepackage{amsmath, amssymb, amsthm, mathtools}
\usepackage{dsfont}
\usepackage{enumitem}
\usepackage[utf8]{inputenc}
\usepackage[T1]{fontenc}
\usepackage{lmodern}

\usepackage[colorlinks=true,linkcolor=blue,citecolor=blue,urlcolor=blue]{hyperref}
\usepackage[alphabetic]{amsrefs}

\theoremstyle{plain}
\newtheorem{theorem}{Theorem}[section]
\newtheorem{lemma}[theorem]{Lemma}
\newtheorem{corollary}[theorem]{Corollary}
\newtheorem{proposition}[theorem]{Proposition}

\theoremstyle{definition}
\newtheorem{definition}[theorem]{Definition}
\theoremstyle{remark}
\newtheorem{remark}[theorem]{Remark}

\newcommand{\densa}{\rho^a}
\newcommand{\dens}{\rho}
\newcommand{\densaeps}{\rho^{a,\varepsilon}}
\newcommand{\denseps}{\rho^{\varepsilon}}

\newcommand{\F}{F}
\newcommand{\ltwo}{L}
\newcommand{\unitint}{(0,1)}
\newcommand{\co}{C_0}

\newcommand{\bb}{\beta}
\newcommand{\rbb}[1][]{|\bb_{#1}|}
\newcommand{\bm}{B}
\newcommand{\mubb}{\mu^{\bb}}
\newcommand{\murbb}{\mu^{\rbb{}}}
\newcommand{\ind}{\mathds{1}}
\newcommand{\indzi}{\ind_{[0,\infty)}}
\newcommand{\inds}{\indzi^{\varepsilon}}

\newcommand{\gmoll}{\gamma_{\varepsilon}}

\newcommand{\DDx}{\delta_0(\bb_x)}
\newcommand{\DDa}{\delta_0(\bb_a)}
\newcommand{\R}{\mathbb{R}}
\newcommand{\C}{\mathbb{C}}
\newcommand{\N}{\mathbb{N}}

\newcommand{\fcbinf}{\mathcal{F}C_b^{\infty}}
\newcommand{\cbone}{C_b^1(\ltwo)}
\newcommand{\HidaFcts}{(\mathcal{S})}
\newcommand{\HidaDistribs}{\HidaFcts'}
\newcommand{\SchwartzFcts}{\mathcal{S}(\R)}
\newcommand{\SchwartzDistribs}{\mathcal{S}'(\R)}
\newcommand{\hc}{\mathrm{HC}}
\newcommand{\ca}{\mathrm{ca}}

\newcommand{\pin}{\mathrm{pin}}
\newcommand{\borel}{\mathcal{B}}
\newcommand{\E}{\mathbb{E}}
\newcommand{\Prob}{\mathbb{P}}
\newcommand{\Emubb}{\E}
\newcommand{\ho}{H_0^1}
\newcommand{\hno}{H_0^{-1}}
\newcommand{\domL}{\ho~\cap~H^2}
\newcommand{\coec}{[\co]}
\newcommand{\Czo}{C(\mkern3.5mu\overline{\mkern-3.5mu (0,1) \mkern-3.5mu}\mkern3.5mu)} 
\newcommand{\Cinfzo}{C^{\infty}(\mkern3.5mu\overline{\mkern-3.5mu (0,1) \mkern-3.5mu}\mkern3.5mu)}
\newcommand{\setza}{\mathcal{G}}
\newcommand{\boundtight}{C_T}
\newcommand{\wick}[1]{\colon\!\! #1 \colon}

\newcommand{\ddext}{\delta_x^{\mathrm{ext}}}
\newcommand{\ddexta}{\delta_a^{\mathrm{ext}}}

\newcommand{\densxbv}{\chi_{\varepsilon}^x}
\newcommand{\vvmeas}{\nu}
\newcommand{\vvmeaseps}{\vvmeas^{\varepsilon}}
\newcommand{\vvmeasx}{\sigma^{(a)}}
\newcommand{\vvmeasepsx}{\sigma^{(a,\varepsilon)}}
\newcommand{\ufd}[1]{n_{#1}}

\newcommand{\var}[1]{\| #1 \|} 
\newcommand{\svar}[1]{\| #1 \|_{\text{SV}}} 
\newcommand{\vmeas}{\var{\vvmeas}}

\newcommand{\vmeasx}{\var{\vvmeasx}}
\newcommand{\ufieldx}{n_{\vvmeasx}}
\newcommand{\lmeasx}{\mu^{(a)}}
\newcommand{\approxmeasx}{\mu^{(a,\varepsilon)}}
\newcommand{\pcaf}[1]{L_{#1}^{\var{\vvmeas}}}
\newcommand{\pcafx}[1]{L_{#1}^{\var{\vvmeasx}}}

\newcommand{\llangle}{\langle \!\langle}
\newcommand{\rrangle}{\rangle \!\rangle}
\newcommand{\bllangle}{\Big\langle \mspace{-7mu} \Big\langle}
\newcommand{\brrangle}{\Big\rangle \mspace{-7mu} \Big\rangle}
\newcommand{\incl}{\mathrel{\reflectbox{$\in$}}}

\newcommand{\supp}{\mathrm{supp}}

\title[On Skorokhod Problems for Reflected and Singular SHE\MakeLowercase{s}]{On Skorokhod Problems for Reflected and Singular Stochastic Heat Equations}
\author[M.\ Grothaus]{Martin Grothaus}
\address{M.\ Grothaus, Department of Mathematics, RPTU University Kaiserslautern--Landau, 67653 Kaiserslautern, Germany}
\email{grothaus@rptu.de}
\urladdr{math.rptu.de/en/wgs/fuana}

\author[N.\ Renner]{Nicolas Renner}
\address{N.\ Renner, Department of Mathematics, RPTU University Kaiserslautern--Landau, 67653 Kaiserslautern, Germany}
\email{rennern@rptu.de}
\urladdr{math.rptu.de/en/wgs/fuana}

\subjclass{60H17, 31C25, 60H07, 60H40, 46G10, 26A45}
\keywords{Stochastic heat equation, reflection or singular drift, integration by parts formulas in infinite dimensions, Dirichlet forms, Malliavin calculus, white noise analysis, vector measures of bounded variation}

\begin{document}

\begin{abstract}
	We prove a Skorokhod decomposition for the Markov processes $X^a$ and $X$ associated to the gradient Dirichlet forms w.r.t.\ the measures $\densa\mubb$ and $\dens\mubb$, respectively. Here, $\mubb$ is the law of the standard Brownian bridge $\bb$, while $\densa$ and $\dens$ denote densities which are given by $\densa(z) \coloneqq \indzi(\bar{z}_a)$ and $\dens(z) \coloneqq \int_0^1 \indzi(\bar{z}_x) \, dx$, respectively, for all $z\in L^2(0,1)$ which have a (unique) continuous representative $\bar{z}$ which vanishes at zero and one. To this end, we derive infinite-dimensional integration by parts formulas (IbPFs) w.r.t.\ $\densa\mubb$ and $\dens\mubb$, which contain Hida distributions alongside the usual drift terms. We represent these Hida distributions by integration w.r.t.\ vector measures of bounded variation. The vector measures in question are constructed via an approximation argument, making use of a generalization of Prokhorov's theorem for vector measures. We further prove that, almost surely, the sample paths of $X^a$ and $X$ take values in the equivalence class of continuous functions vanishing at zero and one for all and $dt$-almost all times, respectively. The main motivation for studying $\densa\mubb$ and $\dens\mubb$ lies in the fact that the distributional terms in their IbPFs are simplifications of the distributional term in the IbPF w.r.t.\ the law of the reflected Brownian bridge on the unit interval $\murbb$, see \cite{GV18}. Representing the latter by integration w.r.t.\ a vector measure of bounded variation is still an open problem.
\end{abstract}

\maketitle

\section{Introduction}
Stochastic partial differential equations (SPDEs) driven by space-time white noise with singular drift terms form a major and developing branch in modern stochastic analysis.
SPDEs of this type naturally arise in quantum field theory (e.g.\ ultraviolet divergences, renormalization), statistical mechanics (e.g. the KPZ equation) and pose challenging mathematical problems due to the divergences caused by the noise term.
For analyzing infinite-dimensional processes, the Fuku\-shima stochastic calculus, relying on Dirichlet form theory \cite{FOT11, MR92}, can be viewed as a partial analog of the powerful and well-established It\^{o} calculus for finite-dimensional processes. Several classes of SPDEs can be treated with this calculus, for example sufficiently regular gradient-type SPDEs, reflected SPDEs, stochastic quantization equations and infinite-dimensional Ornstein--Uhlenbeck processes.
Furthermore, recently developed theories such as the theory of regularity structures \cite{H14}, rough paths \cite{FH14} and paracontrolled distributions \cite{GIP15} have greatly extended the understanding of such singular SPDEs, providing a well-defined notion of solution for SPDEs with locally subcritical nonlinearities.

Infinite-dimensional random obstacle problems give rise to a further, major class of singular SPDEs. Here, the ill-defined drift terms result from reflection at the boundary of the obstacle.
A classical problem in this regard is the dynamic of a one-dimensional random string with Dirichlet boundary conditions on the unit interval, conditioned to stay non-negative.
Solving the Nualart--Pardoux equation \cite{NP92} marks one of the first advances regarding such problems.
More recent work is provided by \cite{RZZ12}, treating the stochastic reflection problem on convex subsets of a Hilbert space with a sufficiently regular boundary by means of Dirichlet form theory.
Within in the Dirichlet form approach, integration by parts formulas (IbPFs) on the considered path spaces play a vital role, as they can be viewed as the infinitesimal versions of the SPDEs solved by the processes associated to the considered Dirichlet forms via Revuz correspondence.
Due to the singular nature of reflections, these IbPFs may contain sophisticated distributional terms which can be studied within the framework of white noise analysis \cite{HKPS93, Oba94, Kuo96}.

Among the stochastic processes with non-negative paths on the unit interval that vanish at both endpoints, a prototypical example is the modulus of the standard Brownian bridge, also called the reflected Brownian bridge.
The rigorous formulation of an SPDE with a well-defined solution whose invariant measure is the law of the reflected Brownian bridge $\murbb$ is still a widely open problem. Resolving it would lead to valuable progress in the understanding of intricate reflection dynamics.
Heuristically, the solution satisfies the KPZ equation (with an added local time) on the boundary. This equation has recently received considerable attention due to breakthroughs via the theory of regularity structures, see \cite{H13}.
In \cite{GV18}, the authors found the IbPF w.r.t.\ $\murbb$, which reads
\begin{equation}\label{GV_IbPF}
	\E_{\murbb}\left[ \partial_h \F \right] + \E_{\murbb}\left[ \F \, (h'',\cdot)_{\ltwo} \right] = \bllangle \F(\rbb), 2\int_{0}^1 h_x \Big(:\Big(\frac{d\bb_x}{dx}\Big)^2: -1 \Big) \, \DDx \, dx \brrangle
\end{equation}
where $h\in C^2(0,1)$ with compact support and $\F \in \fcbinf(L^2(0,1))$, i.e.\ $\F$ is a map from $L^2(0,1)$ to $\R$ of the form
\begin{equation}
	\F(z) = g((\eta_1, z)_{L^2(0,1)}, ..., (\eta_k, z)_{L^2(0,1)}) \quad \text{for all $z\in L^2(0,1)$}
\end{equation}
for some $g\in C_b^{\infty}(\R^k)$, $\eta_i\in L^2(0,1)$ for $i\in \{1,...,k\}$ and $k\in\N$. Further, $\bb$ is the standard Brownian bridge on the unit interval, $:\cdot:$ denotes the Wick renormalization, and $\DDx$ is Donsker's delta w.r.t.\ the Brownian bridge at $x$. Lastly, $\llangle \cdot, \cdot \rrangle$ is the dual pairing in the Hida triple.
This result is inspired by the IbPF on the law of the reflected Brownian motion found in \cite{Zam05}.

Generalizing the case of the reflected Brownian bridge, IbPFs on the law of the $\delta$-dimensional Bessel bridges for $\delta=3$, $\delta>3$ and $0 < \delta<3$ are provided in \cite{Zam02}, \cite{Zam03} and \cite{AZ20}, respectively.
Here, the most challenging and least understood cases are the ones with $\delta < 3$, among which the one-dimensional Bessel bridge, equal to reflected Brownian bridge, is one of the most prominent.

To deepen the understanding of \eqref{GV_IbPF} and the dynamic encoded in it, the question arises whether the Hida distribution on the right hand side (RHS) of the IbPF is regular enough to admit a representation via integration w.r.t.\ a vector measure of bounded variation.
More precisely, one wants to find a $H'$-valued vector measure $V$ on the Borel $\sigma$-algebra $\borel(L^2(0,1))$, where $H'$ is the topological dual space of a suitable, densely embedded Hilbert space $H\subset L^2(0,1)$, s.th.\
\begin{equation*}
	\bllangle \F(\rbb), 2\int_{0}^1 h_x \Big(:\Big(\frac{d\bb_x}{dx}\Big)^2: -1 \Big) \, \DDx \, dx \brrangle
	=
	\Big\langle h, \int_{\ltwo} \F \, dV \Big\rangle
\end{equation*}
for all $h\in H$ and $\F \in \mathcal{F}C_b^1(L^2(0,1))$ (defined analogously to $\fcbinf(L^2(0,1))$ above).

In \cite{AZ20}, the authors approximate the Hida distribution, which emerges in their IbPF for $\murbb$, with measures, and show that the one-potentials associated to these measures converge in the domain of the gradient Dirichlet form w.r.t.\ the law of the standard Brownian bridge.
While this is undoubtedly a notable achievement, it falls short of tackling the representation problem above, which remains a challenging open problem.

In order to progress the understanding around this representation problem, while analyzing interesting cases in their own right, it is reasonable to simplify the Hida distribution in \eqref{GV_IbPF} by omitting the squared renormalized white noise term and even the integral over the unit interval, yielding the RHSs
\begin{equation*}
	\bllangle \F(\bb), \int_{0}^1 h_x \, \DDx \, dx \brrangle
	\qquad \text{and} \qquad
	\bllangle \F(\bb), h_a \, \DDa \brrangle \quad \text{for } a \in (0,1).
\end{equation*}
This directly raises the question of finding measures which induce IbPFs incorporating these simplified Hida distributions, and of studying the properties of the gradient Dirichlet forms w.r.t.\ these measures, and further analyzing their associated processes.
This is the very aim of this paper. Therefore, we now introduce Dirichlet forms which lead to the aforementioned simplified IbPFs, and then solve the corresponding representation problems via vector measures. This finally yields the Skorokhod decompositions of the associated processes.

\medskip

The central Dirichlet forms studied in this paper are
\begin{equation*}
	\mathcal{E}^a(F,G) \coloneqq \int_{\ltwo} (DF(z),DG(z))_{\ltwo} \, \densa \, d\mubb
	\quad \text{and} \quad
	\mathcal{E}(F,G) \coloneqq \int_{\ltwo} (DF(z),DG(z))_{\ltwo} \, \dens \, d\mubb
\end{equation*}
where $\ltwo \coloneqq L^2(0,1)$, $\mubb$ is the law of the standard Brownian bridge $\bb$ on the Borel $\sigma$-algebra $\borel(\ltwo)$, $u$, $v\in W^{1,2}(\ltwo,\mubb)$, and the considered densities $\densa,\, \dens: \ltwo \to \R$ are defined by
\begin{equation*}
	\densa(z) \coloneqq \indzi(\bar{z}_a)\quad \text{and} \quad \dens(z) \coloneqq \int_0^1 \indzi(\bar{z}_x) \, dx \quad \text{for } z \in \coec.
\end{equation*}
Here, $\coec$ is the canonical embedding of $\co$, denoting the continuous functions on the unit interval vanishing at zero and one, into $\ltwo$, and $\bar{z}$ denotes the unique existing $\co$-representative of a function $z\in\ltwo$. Naturally, we extend $\densa$ and $\dens$ to $\ltwo$ by zero. Further, we define the diffusion processes (so called distorted Ornstein--Uhlenbeck processes) associated with $\mathcal{E}^a$ and $\mathcal{E}$ as $M^a = (\Omega^a, \mathcal{M}^a, \{\mathcal{M}_t^a\}, \theta_t^a, X_t^a, P_z^a)$ and $M = (\Omega, \mathcal{M}, \{\mathcal{M}_t\}, \theta_t, X_t, P_z)$, respectively.

As the first key result, we establish the IbPFs
\begin{align*}
	\E_{\densa\mubb}\left[ \partial_h \F \right] + \E_{\densa\mubb}\left[ \F \, (h'',\cdot)_{\ltwo} \right] &= - \llangle \F(\bb), \bar{h}_a \, \DDa \rrangle,\\
	\E_{\dens\mubb}\left[ \partial_h \F \right] + \E_{\dens\mubb}\left[ \F \, (h'',\cdot)_{\ltwo} \right] &= - \bllangle \F(\bb), \int_{(0,1)}\bar{h}_x \, \DDx \, dx \brrangle
\end{align*}
w.r.t.\ the measures $\densa\mubb$ and $\dens\mubb$, respectively, as stated in Theorem \ref{thm:ibpf}. Here, $\F$ is an element of the continuous and bounded functions on $\ltwo$ which have a continuous and bounded Fr\'{e}chet derivative, denoted by $\cbone$. The function $h$ is taken from the domain of the Dirichlet Laplacian on the unit interval, namely the intersection of the Sobolev spaces $\domL \coloneqq H_0^{1,2}(0,1) \cap H^{2,2}(0,1)$.
An instrumental tool for proving these IbPFs consists in a version of the Cameron--Martin formula which allows for $\domL$-displacements, see Lemma \ref{lem:CMF}, which we establish with tools from white noise analysis.
As anticipated above, these IbPFs are indeed simplified versions of the IbPF \eqref{GV_IbPF}.

Moreover, in order to represent the right hand sides (RHSs) of the IbPFs above by integration w.r.t.\ vector measures with bounded variation, and motivated by the $L_+^2(\SchwartzDistribs, \mu)$-approximations
\begin{equation*}
	\gmoll(\bb_a) \xrightarrow[\varepsilon \downarrow 0]{\HidaDistribs} \DDa
	\quad \text{and} \quad
	\int_{(0,1)} \bar{h}_x \gmoll(\bb_x(\omega)) \, dx \xrightarrow[\varepsilon \downarrow 0]{\HidaDistribs} \int_{(0,1)} \bar{h}_x \delta_0(\bb_x) \, dx,
\end{equation*}
we define the $\hno$-valued vector measures $\vvmeasepsx$ and $\vvmeaseps$ on $\borel(\ltwo)$ by
\begin{equation*}
	\vvmeasepsx(h,B) \coloneqq \int_{B} \bar{h}_a \gmoll(\bar{z}_a) \, d\mubb(z)
	\quad \text{and} \quad
	\vvmeaseps(h,B) \coloneqq \int_{B} \int_{(0,1)} \bar{h}_x\gmoll(\bar{z}_x) \, dx \, d\mubb(z)
\end{equation*}
for all $h\in\ho, \, B\in\borel(\ltwo)$. By virtue of a generalization of Prokhorov's theorem for vector measures, we show that the sequences $(\sigma^{(a,1/n)})_{n\in\N}$ and $(\nu^{1/n})_{n\in\N}$ are weak sequentially compact, and we denote arbitrary but fixed weak limits of existing weakly convergent subsequences by $\vvmeasx$ and $\vvmeas$, respectively. These vector measures indeed lead to the desired representation
\begin{equation*}
	\llangle \F(\bb), \bar{l}_a\DDa\rrangle = \Big\langle l, \int_{\ltwo} \F \,d\vvmeasx \Big\rangle
	\quad \text{and} \quad
	\bllangle \F(\bb), \int_{[0,1]} \bar{l}_x\DDx dx \brrangle = \Big\langle l, \int_{\ltwo} \F \,d\vvmeas\Big\rangle,
\end{equation*}
where $l\in \domL$, $\F\in\cbone$ and $\langle\cdot,\cdot\rangle$ denotes the dual pairing in the Gelfand triple $\ho \subset \ltwo \subset \hno$. We denote the polar decomposition as $\vvmeasx = \ufd{\vvmeasx} \var{\vvmeasx}$, where $\var{\vvmeasx}$ is the variation measure and $\ufd{\vvmeasx}$ the unit field associated to $\vvmeasx$.

Now, results from \cite{RZZ12} are applicable to obtain the Skorokhod decompositions of $X^a$ and $X$. Explicitly, for $\mathcal{E}^a$-quasi-every $z\in\ltwo$ there exists an $\mathcal{M}_t^a$-cylindrical Wiener process $W^{z,a}$ such that
\begin{equation*}
	\langle l, X_{t}^a-X_0^a\rangle = \int_0^{t} \langle l, dW_s^{z,a}\rangle + \frac{1}{2} \int_0^{t} \langle l, \ufieldx (X_s^a) \rangle d\pcafx{s} - \int_0^{t} \langle l'', X_s^a \rangle ds
\end{equation*}
for all $l\in\domL$ and $t\geq 0$ $P_z^a$-a.s. Here, $\pcafx{t}$ denotes the real-valued positive continuous additive functional (PCAF) associated with $\vmeasx$ by the Revuz correspondence. The analogous statement holds for $X$. We further prove, by using the strategy of the proof \cite{RZZ12}, Theorem 6.5, that $X^a$ and $X$ are almost surely continuous (more precisely in $\coec$) for quasi-every starting point $z\in\ltwo$, which can be improved to indeed every starting point $z\in\ltwo$ in the case of $X^a$. We also prove further properties of $\var{\vvmeasx}$, $\pcafx{}$ and $X^a$.

We want to note that the processes $X^a$ and $X$ can be heuristically seen as weak solutions to the highly singular SPDEs
\begin{equation*}
	dX_t^a = \frac{1}{2} \Delta X^a_t \, dt + dW_t + \delta_a \otimes dl_t^{0,X_{\cdot}^a(a)}
	\quad \text{and} \quad
	dX_t = \frac{1}{2} \Delta X_t \, dt + dW_t + dl_t^{0,X_{\cdot}},
\end{equation*}
respectively, each with Dirichlet boundary conditions on $\unitint$, where $(l_t^{0,X_{\cdot}^a(a)})_{t\geq 0}$ is a family of local times of $(X_{t}^a(a))_{t\geq 0}$ at zero, $(l_t^{0,X_{\cdot}})_{t\geq 0}$ is a family of local times of $(X_{t})_{t\geq 0}$ at zero.

\medskip

This paper is structured as follows:
\begin{itemize}
	\item  In Section 2, we briefly introduce the main concepts of white noise analysis, define a Brownian bridge in this framework, and describe the basic setting and notations in this paper, including a Gelfand triple and the notion of BV functions therein.
	\item In Section 3, we establish the quasi-regular local gradient Dirichlet forms $\mathcal{E}^a$ and $\mathcal{E}$.
	\item Section 4 proves the IbPFs w.r.t.\ $\densa\mubb$ and $\dens\mubb$. Therefore, we show a version of the Cameron--Martin theorem, and then proceed to prove smoothed versions of the IbPFs.
	\item In Section 5, we represent the RHSs of the IbPFs by integration w.r.t.\ vector measures, leading to the Skorokhod decomposition of $X^a$ and $X$.
	Among other properties, we show that, almost surely, the sample paths of $X^a$ and $X$ is $\coec$-valued for all and for $dt$-almost all times, respectively. Solving the initial representation problem is the central part of this section. It is solved by defining an approximating sequence of vector measures and then utilizing a generalization of Prokhorov's theorem in order to obtain weak limits.
\end{itemize}

\section{Preliminaries}
\subsection{White Noise Analysis}\label{sect:wna}
In this subsection, we provide a brief overview of the main concepts of white noise analysis.
For further reading, we refer to \cites{HKPS93, Oba94, Kuo96}.

First, consider the Gelfand triple
\begin{equation}\label{eqn:SchwartzTriple}
	\SchwartzFcts \subset L^2(\R) \subset \SchwartzDistribs,
\end{equation}
where $\SchwartzFcts$ denotes the space of smooth, rapidly decreasing functions on $\R$ (which is the projective limit of the Hilbert spaces $(\mathcal{H}_p, \|\cdot\|_p)$, $p\in \N_{0}$), and $\SchwartzDistribs$ denotes the topological dual space of $\SchwartzFcts$, called the tempered distributions. As usual, in \eqref{eqn:SchwartzTriple} we identify $L^2(\R)$ with its topological dual space, which is justified by the Riesz representation theorem. We denote the dual pairing on $\SchwartzFcts \times \SchwartzDistribs$ by $\langle \cdot, \cdot \rangle$. Further, we also make use of the corresponding complexified spaces $\mathcal{S}_{\C}(\R)$, $L_{\C}^2(\R)$ and $\mathcal{S}_{\C}'(\R)$, as well as the corresponding bilinear extension of the dual pairing. Now, the Bochner--Minlos theorem implies that there exists a unique probability measure $\mu$ on $(\SchwartzDistribs, \borel)$, where $\borel$ denotes the cylindrical $\sigma$-algebra on $\SchwartzDistribs$, s.th.
\begin{equation*}
	\int_{\SchwartzDistribs} \exp(i\langle \varphi, \omega \rangle) d\mu(\omega) = \exp\left(-\frac{1}{2} (\varphi, \varphi)_{L^2(\R)} \right) \text{ for all $\varphi \in \SchwartzFcts$}.
\end{equation*}
We call $\mu$ the \emph{white noise measure} and define $L^2(\mu) \coloneqq L^2(\SchwartzDistribs, \borel, \mu; \C)$. Using the fact that the linear map $\mathcal{S}_{\C}(\R) \to L^2(\mu): \eta \mapsto \langle \eta, \cdot \rangle$ is isometric and that $\mathcal{S}_{\C}(\R)$ is dense in $L_{\C}^2(\R)$, one can define $\langle f, \cdot \rangle$ as an element of $L^2(\mu)$ for any $f\in L_{\C}^2(\R)$ in a straightforward manner, and extend the dual pairing $\langle \cdot, \cdot \rangle$ to $L_{\C}^2(\R) \times \mathcal{S}_{\C}'(\R)$ in this sense. For this extension, the isometry above remains valid, i.e.\
\begin{equation}\label{eqn:wna_isometry}
	\| \langle f, \cdot \rangle \|_{L^2(\mu)} = \| f \|_{L_{\C}^2(\R)} \text{ for all $f\in L_{\C}^2(\R)$.}
\end{equation}
Further, one can show that for $f_1,...,f_n \in L^2(\R)$, $n\in\N$, the image measure of $\mu$ under the map $P_{f_1,...,f_n}: \SchwartzDistribs \to \R^n: \omega \mapsto (\langle f_i, \omega \rangle)_{1\leq i \leq n}$ is the Gaussian measure with mean zero and covariance matrix $\Lambda = ((f_i,f_j)_{L^2(\R)})_{1\leq i,j \leq n}$ on $\R^n$, i.e.\
\begin{equation}\label{eqn:wna_image_meas_lem}
	P_{f_1,...,f_n}(\mu) = \mathcal{N}(0,\Lambda).
\end{equation}
It also holds 
\begin{equation}\label{eqn:wna_int_exp}
	\int_{\SchwartzDistribs} \exp\left( \langle f, \omega\rangle \right) d\mu(\omega) = \exp\left(\frac{1}{2} \langle f, f\rangle\right) \text{ for all $f\in L_{\C}^2(\R)$}.
\end{equation}
Moreover, for every fixed $\omega_0 \in \SchwartzDistribs$, we define the translation
\begin{equation*}
	T_{+\omega_0}^{\mathcal{S}'} : \SchwartzDistribs \to \SchwartzDistribs: \omega \mapsto \omega + \omega_0,
\end{equation*}
which is continuous (w.r.t.\ the weak$^*$-topology on $\SchwartzDistribs$) and $\borel$-measurable. This is crucial for the following result:

\begin{proposition}[Cameron--Martin Formula in White Noise Analysis]\label{prop:cmf_wna}
	For any $f \in L^2(\R)$, the image measure $T_{+f}^{\mathcal{S}'}(\mu)$ of $\mu$ under the translation $T_{+f}^{\mathcal{S}'}$ is absolutely continuous w.r.t.\ $\mu$, and its Radon--Nikodym derivative is given by
	\begin{equation*}
		\frac{dT_{+f}^{\mathcal{S}'}(\mu)}{d\mu} = \exp\left( \langle f, \cdot \rangle - \frac{1}{2} \| f \|_{L^2(\R)} \right).
	\end{equation*}
\end{proposition}
Similarly to \eqref{eqn:SchwartzTriple}, one can construct the Gelfand triple
\begin{equation*}
	\HidaFcts \subset L^2(\mu) \subset \HidaDistribs,
\end{equation*}
where $\HidaFcts$ is the space of Hida test functions (which is the projective limit of the Hilbert spaces $((\mathcal{H}_p), |\!|\!|\cdot|\!|\!|_p)$, $p\in \N_{0}$), and where $\HidaDistribs$ denotes its topological dual space, called the space of Hida distributions. Let $\llangle \cdot, \cdot \rrangle$ be the dual pairing on $\HidaFcts \times \HidaDistribs$, and define the $\mathrm{S}$- and $\mathrm{T}$-transform of an element $\Phi\in\HidaDistribs$ by
\begin{equation*}
	\mathrm{S}\Phi(\varphi) \coloneqq \llangle \wick{\exp(\langle \varphi, \cdot \rangle)}, \Phi \rrangle
	\text{, }
	\mathrm{T}\Phi(\varphi) \coloneqq \llangle \exp(\langle i\varphi, \cdot \rangle), \Phi \rrangle
	 \text{ for all } \varphi\in\mathcal{S}_{\C}(\R),
\end{equation*}
where we used the \emph{Wick exponential} given by
\begin{equation*}
	\wick{\exp(\langle \varphi, \cdot \rangle)} \coloneqq \exp\left(\langle \varphi, \cdot \rangle - \frac{1}{2}\langle \varphi, \varphi \rangle\right).
\end{equation*}
Note that the $\mathrm{S}$- and $\mathrm{T}$-transform are connected by the relation
\begin{equation*}
	\mathrm{T}\Phi(\varphi) = \exp\left(-\frac{1}{2}\langle \varphi, \varphi \rangle\right) \mathrm{S}\Phi(i\varphi).
\end{equation*}

\begin{definition}[$U$-Functionals]
	A function $U: \SchwartzFcts \to \C$ is called a \emph{$U$-functional} iff $U$ meets both of the following two conditions:
	\begin{enumerate}
		\item $U$ is \emph{ray-analytic}, i.e.\ for all $\varphi$, $\psi \in \SchwartzFcts$ the function $\R \incl s \mapsto U(\varphi + s \psi) \in \C$
		is entire analytic and therefore extends to an entire function on $\C$.
		\item $U$ is \emph{uniformly bounded of exponential order 2}, i.e.\ there exist constants $A$, $B\in\R_{\geq 0}$ and $p \in \N_{0}$ s.th.\ for all $z\in\C$, $\varphi \in \SchwartzFcts$ it holds
		\begin{equation*}
			|U(z\varphi)| \leq A \exp(B|z|^2\|\varphi\|_{p}^2).
		\end{equation*}
	\end{enumerate}
\end{definition}

The following theorem is crucial for handling Hida distributions, since it allows us to analyze Hida distributions by means of their corresponding $U$-functionals:

\begin{theorem}[Characterization of Hida Distributions]
	Both $\mathrm{S}$- and $\mathrm{T}$-transform are $\C$-linear bijections from the Hida space $\HidaDistribs$ to the space of $U$-functionals.
\end{theorem}

\begin{theorem}[Criterion for Convergence in $\HidaDistribs$]\label{thm:conv_in_hidaspace}
	Let $(\Phi_n)_{n\in\N}$ be a sequence in $\HidaDistribs$. If the corresponding sequence of $U$-functionals $(U_n)_{n\in\N} \coloneqq (\mathrm{S}^{-1}\Phi_n)_{n\in\N}$ satisfies both of the following two conditions
	\begin{enumerate}
		\item $(U_n(\varphi))_{n\in\N}$ is a Cauchy sequence for all $\varphi \in \SchwartzFcts$\text{, and}
		\item there exist $p$, $N\in\N$ and $A$, $B\in\R_{\geq0}$ s.th.\
		\begin{equation*}
			|U_n(z\varphi)| \leq A \exp(B|z|^2\|\varphi\|_p^2)\quad \text{for all $z\in\C$, $\varphi \in \SchwartzFcts$ and $n\geq N$}.
		\end{equation*}
	\end{enumerate}
	Then, the point-wise limit $U\coloneqq \lim_{n \to \infty} U_n$ is a $U$-functional, and we have $\Phi_n \to \mathrm{S}^{-1}(U) \eqqcolon \Phi$ strongly in $\HidaDistribs$ as $n\to\infty$. The analogous statement holds if the $\mathrm{S}$-transform is exchanged for the $\mathrm{T}$-transform.
\end{theorem}

\begin{theorem}[Criterion for Bochner-Integrability in $\HidaDistribs$]\label{thm:bochner_int_in_hidaspace}
	Let $(\Omega, \mathcal{A}, m)$ be a measure space. If the map $\Phi:\Omega \to \HidaDistribs: \lambda \mapsto \Phi_{\lambda}$ satisfies that
	\begin{enumerate}
		\item for all $\varphi\in\SchwartzFcts$ the map $\Omega \incl \lambda \mapsto \mathrm{S}\Phi_{\lambda}(\varphi) \in\C$ is measurable, and that
		\item there exist $A\in \mathcal{L}_+^1(\Omega,m;\overline{\R})$, $B\in \mathcal{L}_+^{\infty}(\Omega,m;\overline{\R})$ and $p\in\N_{0}$ s.th.\ for all $\lambda\in\Omega$, $z\in\C$ and $\varphi\in\SchwartzFcts$ it holds
		\begin{equation*}
			|\mathrm{S}\Phi_{\lambda}(z\varphi)| \leq A(\lambda) \exp(B(\lambda)|z|^2\|\varphi\|_p^2),
		\end{equation*}
	\end{enumerate}
	then there exists some $p'\in\N$ s.th.\ $\Phi\in L^1(\Omega,m; (\mathcal{H}_{-p'}))$, i.e.\ $\Phi$ is Bochner-integrable with values in $(\mathcal{H}_{-p'})$ w.r.t.\ $m$. In this case, it holds in particular that
	\begin{equation*}
		\int_{\Omega} \Phi_{\lambda} dm(\lambda) \in (\mathcal{H}_{-p'}) \subset \HidaDistribs,
	\end{equation*}
	and taking the $\mathrm{S}$-transform commutes with the Bochner-integration, i.e.\ for all $\varphi \in \SchwartzFcts$ we have
	\begin{equation*}
		\left( \mathrm{S} \left( \int_{\Omega} \Phi_{\lambda} dm(\lambda) \right) \right)(\varphi)
		= \int_{\Omega} \mathrm{S}\Phi_{\lambda}(\varphi) dm(\lambda).
	\end{equation*}
	The analogous statement holds if the $\mathrm{S}$-transform is exchanged for the $\mathrm{T}$-transform.
\end{theorem}

\subsection{The Brownian Bridge}\label{sect:bb}
This subsection introduces our handling of the Brownian bridge, based on the approach taken in \cite{GV18}, p.\ 341 and Subsection 5: As seen in Subsection \ref{sect:wna}, $(\langle \ind_{[0,x)}, \cdot \rangle)_{x>0}$ is a family of well-defined elements of $L^2(\mu)$. By an application of the Kolmogorov--\u{C}ensov--Lo\`{e}ve theorem, it can be shown that its existing continuous modification is a standard Brownian motion on $(\SchwartzDistribs, \borel, \mu)$, which we denote by $(\bm_x)_{x\geq 0}$, where $\bm_0 = 0$. Hence,
\begin{equation*}
    \bb_x \coloneqq \bm_x - x\bm_1 \quad\text{for all $x\in(0,1]$ and}\quad \bb_0 \coloneqq 0
\end{equation*}
defines a Brownian bridge $\bb = (\bb_x)_{x\in [0,1]}$ from 0 to 0. Let
\begin{equation*}
    q_x \coloneqq \ind_{[0,x)}-x\ind_{[0,1)} \text{\, for all \,} x \in (0,1] \text{\, and \,} q_0 \coloneqq 0.
\end{equation*}
Then $\bb_x = \langle q_x, \cdot \rangle$ in $L^2(\mu)$ for all $x\in [0,1]$; in particular, for each $x\in [0,1]$, there exists a $\mu$-null set $N_x$ (depending on $x$) s.th.\ $\bb_x(\omega) = \langle q_x, \omega \rangle$ for all $\omega\in \SchwartzDistribs\setminus N_x$. We make use of the known facts that $\bb_x \sim \mathcal{N}(0,x(1-x))$ for all $x\in [0,1]$ and $\mathrm{Cov}(\bb_x,\bb_y)=x\wedge y - xy$ for all $x$, $y\in [0,1]$. Further, it is natural to identify the process $\bb$ with the corresponding path-valued map
\begin{equation*}
	\bb: \SchwartzDistribs \to \ltwo: \omega \mapsto \bb(\omega) \coloneqq [(\bb_x(\omega))_{x\in (0,1)}],
\end{equation*}
where $\ltwo = L^2(0,1)$, which is a well-defined random variable, since $(\bb_x)_{x\in (0,1)}$ has surely continuous paths, and $\bb$ is measurable w.r.t.\ $\borel(\ltwo)$. The latter fact follows, since $\bb$ is well-defined as a $\co$-valued mapping, the canonical embedding $\co \hookrightarrow \ltwo$ is continuous and $\co$-valued random variables are exactly given by the real-valued processes on $\unitint$ with continuous paths, see \cite{Kal97}, Lemma 14.1.

Consequently, we define $\mubb$ to be the image measure of the white noise measure under the map $\bb$, i.e.\ $\mubb \coloneqq \bb(\mu)$. This means that $\mubb$ is the law of a Brownian bridge from 0 to 0 on $\ltwo$, or equivalently, a centered Gaussian measure on $\ltwo$ with covariance operator $Q: \ltwo \to \ltwo$ given by
\begin{equation}\label{def:cov_operator}
	(Q\eta)_x \coloneqq \int_0^1 (x\wedge y - xy) \, \eta_y \, dy \quad \text{for all } x\in\unitint,\, \eta\in\ltwo.
\end{equation}

\begin{remark}\label{rem:embeddings_and_measurability}
	Let $\coec$ be the canonical embedding of $\co$ into $\ltwo$, and let $\bar{z}$ denote the unique existing $\co$-representative of a function $z\in\coec$. This embedding allows us to leverage the Hilbert space structure of $\ltwo$.
	Further, we view $\co$ as a Banach space, equipped with the supremum norm $\|\cdot\|_{\infty}$, $\ltwo$ as a Hilbert space, equipped with the canonical scalar product $(\cdot,\cdot)_{\ltwo}$ and $\coec$ as a normed subspace of $\ltwo$ equipped with the restricted norm $\|\cdot\|_{\ltwo}\rvert_{\coec}$. Let $\iota: \co \hookrightarrow \ltwo$ denote the canonical inclusion of $\co$ into $\ltwo$. All of these spaces can be viewed as measurable spaces, equipped with the induced Borel $\sigma$-algebras; these measure spaces are compatible in the sense that $\iota\rvert_{\coec}$ is a Borel isomorphism and $\borel(\coec) \subset \borel(\ltwo)$ which can be proven using \cite{Par67}, p.\ 135, Theorem 2.4. In particular it holds $\coec \in \borel(\ltwo)$.
\end{remark}

\subsection{Donsker's Delta}\label{sect:DonskersDelta}
The following is informed by \cite{GV18}, Example 2.6: For $a\in\R$, $0\neq \eta \in L^2(\R)$ we define \emph{Donsker's delta} as
\begin{equation*}
	\delta_a(\langle \eta, \cdot \rangle) \coloneqq \frac{1}{2\pi} \int_{\R} \exp(i(\langle\eta,\cdot\rangle-a) x) \, dx \in \HidaDistribs.
\end{equation*}
This expression is motivated by the Fourier representation of the Dirac measure $\delta_a$, and well-defined by Lemma \ref{thm:bochner_int_in_hidaspace} with $(\Omega,\mathcal{F},m)=(\R,\borel(\R),dx)$.

Considering the Brownian bridge $(\bb_x)_{x\in (0,1)}$, it holds
\begin{align}
	\mathrm{S}\delta_0(\bb_x)(\varphi) &= \frac{1}{\sqrt{2\pi x(1-x)}} \exp\left( -\frac{1}{2x(1-x)} \left(\int_0^1 q_x(s) \varphi_s ds \right)^2 \right),\label{eqn:Strafo_DD}
\end{align}
for all $x\in\unitint$ and $\varphi\in\SchwartzFcts$. In use of the approximate identity $(\gmoll)_{\varepsilon>0}$ given by the Gaussian densities $\gmoll \coloneqq (2\pi \varepsilon)^{-1/2} \exp(-(\cdot)^2/(2\varepsilon))$, one can obtain the approximation
\begin{equation}\label{eqn:DDforBM_approx}
	L_+^2(\SchwartzDistribs, \mu) \incl \gmoll (\bb_x) \to \delta_0(\bb_x) \text{ in $\HidaDistribs$ as $\varepsilon \downarrow 0$}.
\end{equation}

\subsection{The Considered Gelfand Triple and Bounded Variation}\label{sect:gelf_triple_bv}
In this paper, we consider a special case of the setting in \cite{RZZ12}, Section 2, largely equivalent to that in \cite{RZZ12}, Section 6, which we describe in this subsection.

Let $\ltwo$ denote the real separable Hilbert space $L^2(0,1)$ with the canonical scalar product $(\cdot,\cdot)_{\ltwo}$ and the induced norm $\|\cdot\|_{\ltwo}$. When viewed as a measurable space, $\ltwo$ is endowed with its Borel $\sigma$-algebra $\borel(\ltwo)$. Further, we define the linear operator $A$ on $\ltwo$ to be the Dirichlet Laplacian on the one-dimensional open unit interval $(0,1)$ (scaled by $1/2$) given by
\begin{equation*}
	A \coloneqq -\frac{1}{2} \frac{d^2}{dx^2}: D(A)\subset \ltwo \to \ltwo,\text{ where }D(A) \coloneqq \ho \cap H^2,
\end{equation*}
using the simplified notations $\ho \coloneqq H_0^{1,2}(0,1)$ and $H^2\coloneqq H^{2,2}(0,1)$. One can check that $A$ meets the requirements formulated in \cite{RZZ12}, Hypothesis 2.1. Further, using methods from Fourier analysis, one can show that $\ltwo$ has a complete orthonormal basis of eigenvectors of $A$ given by
\begin{equation*}
	(e_j)_{j\in\N} \coloneqq (\sqrt{2}\sin(j\pi \cdot))_{j\in\N}, \text{ where } Ae_j = \frac{(j\pi)^2}{2} e_j \text{ for all } j\in \N.
\end{equation*}
Since $A$ is strictly positive, the measure $\mubb$ is non-degenerate and has full topological support. Recalling \eqref{def:cov_operator}, it is straightforward to check that $A^{-1} = 2Q$.
Further, one can show that for $c_j \coloneqq \|e_j\|_{\ho} = (1+(j\pi)^2)^{1/2}\geq 1$, $j\in\N$,
\begin{equation*}
	H_1 \coloneqq \{f \in \ltwo \mid \|f\|_{H_1} <\infty\} \text{, where } (x,y)_{H_1} = \sum_{j=1}^{\infty} c_j^2 (x,e_j)_{\ltwo} (y,e_j)_{\ltwo},
\end{equation*}
forms a Hilbert space which is isomorphic to $(\ho, (\cdot,\cdot)_{\ho})$. To define the intended notion of bounded variation, the function space
\begin{equation*}
	L(\log L)^{1/2} (\ltwo, \mubb) \coloneqq \{f:\ltwo\to\R \mid \text{$f$ Borel measurable, } A_{1/2}(|f|) \in L^1(\ltwo,\mubb) \}
\end{equation*}
plays an important role, where $A_{1/2}(x) \coloneqq \int_{0}^{x} (\log (1+s))^{1/2} ds$ for all $x \geq 0$.

\medskip

Let $D\F:\ltwo \to \ltwo$ denote the Fr\'{e}chet-derivative of a function $\F:\ltwo\to\R$ (if it exists). Further, let $C_b^1(\ltwo)$ be the set of all bounded differentiable functions with continuous and bounded derivatives.

We introduce a family of $\ltwo$-valued functions on $\ltwo$ by
\begin{equation*}
	(C_b^1)_{\ho\cap H^2}
	\coloneqq
	\Big\{G \;\Big|\; G(z) = \sum_{j=1}^{m}g_j(z)l^j, \; z\in\ltwo, \; g_j\in C_b^1(\ltwo), \; l^j\in \ho\cap H^2 \Big\},
\end{equation*}
and further, let $D^*$ be the adjoint of $D:C_b^1(\ltwo)\subset L^2(\ltwo,\mubb)\to L^2(\ltwo,\mubb;\ltwo).$ It holds $(C_b^1)_{\ho\cap H^2} \subset \mathrm{Dom}(D^*)$, and for every $G=\sum_{j=1}^{m}g_j(\cdot)l^j \in (C_b^1)_{\ho\cap H^2}$ it holds, considering \cite{RZZ12}, Equation (3.5) and linearity of $D^*$, that
\begin{equation}\label{eqn:D*on_class_L2vald_fcts}
	D^*G(z) = \sum_{j=1}^{m} \left( -\langle l^j, Dg_j(z)\rangle +2g_j(z) \langle Al^j, z\rangle \right) \quad \text{for all } z\in\ltwo.
\end{equation}
For any $\dens \in L(\log L)^{1/2} (\ltwo, \mubb)$, we define
\begin{equation*}
	V(\dens) \coloneqq \sup_{G\in (C_b^1)_{\ho\cap H^2}, \|G\|_{\ho}\leq 1} \int_{\ltwo} D^*G(z)\dens(z) d\mubb (z).
\end{equation*}
A function $\dens$ on $\ltwo$ is called a \emph{BV function in the Gelfand triple $(\ho, \ltwo, \hno)$}, notated as $\dens \in \mathrm{BV}(\ltwo, \ho)$, iff $\dens \in L(\log L)^{1/2} (\ltwo, \mubb)$ and $V(\dens) < \infty$.

\subsection{A Sobolev Space on $\ltwo$ and Dense Subspaces}
Let $W^{1,2}(\ltwo;\mubb) \subset L^2(\ltwo, \mubb)$ be the Sobolev space associated to $\mubb$, i.e.\ the closure of $C_b^1(\ltwo)$ w.r.t.\ the norm given by
\begin{equation*}
	\| F \|_{W^{1,2}(\ltwo;\mubb)}^2 = \| F \|_{L^2(\ltwo;\mubb)}^2 + \frac{1}{2} \int_{\ltwo} \|DF\|_{\ltwo}^2 d\mubb \quad \text{for all } F \in C_b^1(\ltwo).
\end{equation*}
Indeed, $W^{1,2}(\ltwo;\mubb)$ is a dense subspace of $L^2(\ltwo,\mubb)$. Further, set
\begin{gather*}
	\exp(C^{\infty}) \coloneqq \mathrm{span}_{\R} \{ \sin((\eta, \cdot)_{\ltwo}), \cos((\eta, \cdot)_{\ltwo}) \mid \eta\in C^{\infty}[0,1] \},\\
	\mathcal{F}C_b^{\infty}(\ltwo) \coloneqq \big\{ g((\eta_1, \cdot)_{\ltwo}, ..., (\eta_k, \cdot)_{\ltwo}) \mid g \in C_b^{\infty}(\R^k),\, k\in\N,\, \eta_1, ..., \eta_k \in \ltwo \big\},
\end{gather*}
where $C_b^{\infty}(\R^k)$ denotes the set of all infinitely differentiable functions on $\R^k$ s.th.\ all partial derivatives are bounded. It holds that
\begin{equation*}
	\exp(C^{\infty}) \subset \mathcal{F}C_b^{\infty}(\ltwo) \subset C_b^1(\ltwo) \subset W^{1,2}(\ltwo;\mubb) \subset L^2(\ltwo;\mubb),
\end{equation*}
and it is known that $\exp(C^{\infty})$ is dense in $W^{1,2}(\ltwo;\mubb)$, see \cite{GV18}, Theorem 5.2.

\subsection{Miscellaneous}\label{sect:misc_intro}
We frequently use the fact that $H^{1,2}(0,1)$ can be continuously embedded into $C_b(0,1)$, i.e.\ the set of bounded and continuous functions on $(0,1)$. More precisely, for every $h\in H^{1,2}(0,1)$, there exists a unique continuous and bounded representative $\bar{h}$ of $h$, and it holds $\|\bar{h}\|_{\infty} \leq C\, \|h\|_{H^{1,2}(0,1)}$ for some constant $C<\infty$ independent of $h$, see \cite{AF03}, Section 4. Further, it holds that for any $h\in\ho$, the unique existing continuous representative $\bar{h}$ can get uniquely extended to $[0,1]$ by zero, see \cite{AF03}, Section 3. The considerations above immediately yield that $\delta_x\in\hno$.
Moreover, for every $x\in (0,1)$, we define $\ddext(z) \coloneqq \bar{z}_x$ for all $z\in\coec$ and extend this map to $\ltwo$ by zero.

\section{The Considered Dirichlet Forms}
In what follows, we fix $a\in (0,1)$ and consider the symmetric bilinear forms given by
\begin{align*}
	\mathcal{E}^a(F,G) &\coloneqq \int_{\ltwo} (D F(z),D G(z))_{\ltwo} \, \densa(z) \, d\mubb(z) \quad \text{for } F,\, G \in \cbone,\\
	\text{and} \quad
	\mathcal{E}(F,G) &\coloneqq \int_{\ltwo} (D F(z),D G(z))_{\ltwo} \, \dens(z) \, d\mubb(z)\quad \text{for } F,\, G \in \cbone,
\end{align*}
where the considered densities $\densa,\, \dens: \ltwo \to \R$ are defined by
\begin{equation}\label{def:densa}
	\densa(z) \coloneqq \indzi(\bar{z}_a)\quad \text{and} \quad \dens(z) \coloneqq \int_0^1 \indzi(\bar{z}_x) \, dx \quad \text{for } z \in \coec,
\end{equation}
respectively, and extended to $\ltwo$ by zero. The central aim of this section is to prove the following theorem:

\begin{theorem}\label{thm:df_closable_qreg}
	The symmetric bilinear forms $(\mathcal{E}^a, C_b^1(\ltwo))$, and $(\mathcal{E}, C_b^1(\ltwo))$ are closable on $L^2(\ltwo, \rho^a \mubb)$ and $L^2(\ltwo, \rho \mubb)$, respectively. Furthermore, their closures $(\mathcal{E}^a, \mathcal{F}^a)$ and $(\mathcal{E}, \mathcal{F})$ are quasi-regular local Dirichlet forms on $L^2(\ltwo, \rho^a \mubb)$ and $L^2(\ltwo, \rho \mubb)$, respectively.
\end{theorem}

In order to prove Theorem \ref{thm:df_closable_qreg}, we make use of \cite{RZZ12}, Theorem 2.2. In the following, we check the assumptions of the aforementioned theorem.

\begin{remark}(Ray Hamza Condition)\label{rem:RayHamza}
	We say that a non-negative measurable function $h:\R \to \R$ possesses the \emph{Hamza property} iff
	\begin{equation*}
		h\rvert_{\R\setminus R(h)} = 0 \text{ $ds$-a.s., where } R(h) \coloneqq \Big\{x\in\R \, \Big| \, \exists \, \varepsilon > 0: \int_{B_{\varepsilon}(x)} \frac{1}{h(s)} \, ds < \infty \Big\}.
	\end{equation*}
	A function $\gamma \in L_+^1(\ltwo,\mubb)$ is said to satisfy the \emph{ray Hamza condition in direction $l\in\ltwo$} (written as $\gamma\in\textbf{H}_l$) iff there exists a non-negative function $\tilde{\gamma}_l$ s.th.\ $\tilde{\gamma}_l = \gamma$ $\mubb$-a.e.\ and $\R \incl s \mapsto \tilde{\gamma}_l(z+sl)\in\R$ possesses the Hamza property for all $z\in\ltwo$. We define $\textbf{H}\coloneqq \bigcap_j \textbf{H}_{e_j}$, where $(e_j)$ is as in Subsection \ref{sect:gelf_triple_bv}, and say that $\gamma \in L_+^1(\ltwo,\mubb)$ satisfies the \emph{ray Hamza condition} iff $\gamma \in \textbf{H}$. It holds $\textbf{H} \subset \mathrm{QR}(\ltwo)$, see \cite{AR90}.
\end{remark}

\begin{lemma}\label{lem:L1+_RayHamza}
	It holds $\densa,\, \dens \in L_+^1(\ltwo, \mubb)$ and $\densa, \dens$ satisfy the ray Hamza condition.
\end{lemma}
\begin{proof}
	It holds $\densa \in L_+^1(\ltwo, \mubb)$, since the point evaluation $\co \incl \psi \mapsto \psi_a \in \R$ is $\borel(\co)$-measurable and $0\leq \densa \leq 1$. Further, $\dens$ is clearly well-defined. The evaluation map $[0,1]\times\co \incl (x,z) \mapsto z_x \in \R$ is continuous, hence $[0,1]\times\co \incl (x,z) \mapsto \indzi(z_x) \in \R$ is measurable and therefore Tonelli's theorem implies that $\co\incl z \mapsto \dens(z)\in\R$ is measurable, which yields that $\dens$ is measurable. Since $0\leq \dens \leq 1$, it holds $\dens \in L_+^1(\ltwo, \mubb)$.
	
	For verifying the ray Hamza condition, let $j\in\N$ and consider that $e_j=\sqrt{2}\sin(j\pi \cdot) \in \coec$. If $z\in\coec^c$, then $z+\R e_j \subset \coec^c$ and hence both maps $s \mapsto \rho^a(z+se_j)$ and $s \mapsto \rho(z+se_j)$ are zero on $\R$ and thus satisfy the Hamza condition. On the contrary, if $z\in\coec$, then $z+\R e_j \subset \coec$ and the following holds: The map
	\begin{equation*}
		\R \incl s \mapsto \rho^a(z+se_j) = \indzi(\bar{z}_a + se_j(a))= \begin{cases}
			\ind_{[r, \infty)}(s), &e_j(a)>0,\\
			\ind_{(-\infty,r]}(s), &e_j(a)<0,\\
			\indzi(\bar{z}_a), &e_j(a)=0,
		\end{cases}
	\end{equation*}
	where $r=-\bar{z}_a/e_j(a)$, satisfies the Hamza condition, since this map is monotone in either case. Further, observe that the map $s \mapsto \rho(z+se_j)$ is the sum of the two maps
	\begin{equation*}
		\R\incl s \mapsto f_{\pm}(s) \coloneqq dx\big(\{ z+se_j \geq 0 \} \cap \{\pm e_j > 0\}\big) \in \R.
	\end{equation*}
	Note that $f_+$ and $f_-$ are monotone and therefore satisfy the Hamza condition. Consequently, their sum also satisfies the Hamza condition. This finishes the proof.
\end{proof}

\begin{lemma}\label{lem:supp_meas_dens_rho}
	The measures $\densa \mubb,\, \rho \mubb$ have full topological support.
\end{lemma}
\begin{proof}
	Let $f\in\ltwo$ and $\varepsilon>0$ arbitrary but fixed. Since $[\Czo]$ is dense in $\ltwo$, we can choose some $\psi\in [\Czo]$ s.th.\ $\psi\in B_{\varepsilon/3}^{\ltwo}(f)$. By a ``small modification'' of $\psi$ around $x \in \{0,\, a,\, 1\}$, it is technical but straightforward to construct a function $\varphi\in \coec \cap B_{\varepsilon/3}^{\ltwo}(\psi)$ s.th.\ $\varphi\rvert_{B_{\delta}(a)} \geq 1$ for some $\delta>0$ small enough. 
	This implies that $B_{\sim}\coloneqq\iota\big(B_{((\varepsilon/3) \wedge 1)}^{\co}(\bar{\varphi})\big) \subseteq B_{\varepsilon/3}^{\ltwo}(\varphi) \subseteq B_{\varepsilon}^{\ltwo}(f)$, which yields that
	\begin{equation*}
		\densa \mubb(B_{\varepsilon}^{\ltwo}(f))
		\geq \densa \mubb\left(B_{\sim}\right) = \mubb\left(B_{\sim}\right) > 0,
	\end{equation*}
	where the equality is due to $\xi\in B_{\sim} \Rightarrow \bar{\xi}_a>0$, and the last estimate holds, since the law of the Brownian bridge on $\co$ has full topological support.
	
	Further, we can use this construction to estimate
	\begin{equation*}
		\rho\mubb(B_{\varepsilon}^{\ltwo}(f))
		\geq \rho\mubb\left(B_{\sim}\right) \geq 2\delta \mubb\left(B_{\sim}\right) > 0,
	\end{equation*}
	where the second estimate follows due to $\xi \in B_{\sim} \Rightarrow \bar{\xi}\rvert_{B_{\delta}(a)} > 0 \Rightarrow \rho(\xi)\geq 2\delta$. This finalizes the proof.
\end{proof}

\begin{proof}[Proof of Theorem \ref{thm:df_closable_qreg}]
	Due to Lemma \ref{lem:L1+_RayHamza} and Lemma \ref{lem:supp_meas_dens_rho}, Theorem \ref{thm:df_closable_qreg} follows by an application of \cite{RZZ12}, Theorem 2.2.
\end{proof}

Now, by Theorem \ref{thm:df_closable_qreg} and the theory presented in \cite{MR92}, there exists a diffusion process $M = (\Omega, \mathcal{M}, \{\mathcal{M}_t\}, \theta_t, X_t, P_z)$ on $\ltwo$ associated with the Dirichlet form $(\mathcal{E}, \mathcal{F})$, called a distorted Ornstein--Uhlenbeck process on $\ltwo$. For $(\mathcal{E}^a, \mathcal{F}^a)$, we denote the corresponding process by $M^a = (\Omega^a, \mathcal{M}^a, \{\mathcal{M}_t^a\}, \theta_t^a, X^a_t, P^a_z)$.

\section{The Associated Integration by Parts Formulas}
\subsubsection{A Version of the Cameron--Martin Formula}
In the following, let $T_{+z_0} : \ltwo \to \ltwo : z \mapsto z_0+z$ denote the continuous and therefore measurable translation in $\ltwo$ by $z_0 \in \ltwo$. We further introduce the operator $K: \ltwo \to \ltwo$ by $(K\varphi)_x \coloneqq (q_{\cdot}(x), \varphi)_{\ltwo}$ for all $x \in \unitint,\, \varphi \in \ltwo$, as discussed in \cite{GV18}, Subsection 4. Note that $K$ is linear and bounded, and its adjoint $K^*$ is given by $(K^* \psi)_y \coloneqq (q_y(\cdot), \psi)_{\ltwo}$ for all $x \in \unitint,\, \psi \in \ltwo$. These operators are useful for the upcoming proofs, since for all $\varphi \in C[0,1]$ it holds
\begin{equation}\label{eqn:L2WNA_translation_formula}
	(\varphi, \bb(\omega))_{\ltwo} = \langle K\varphi, \omega \rangle \quad \text{for $\mu$-a.e.\ $\omega \in \SchwartzDistribs$,}
\end{equation}
where we extended $K\varphi$ to $\R$ by zero, see \cite{GV18}, p.\ 352. Further, note that the fundamental theorem of calculus for Sobolev spaces yields
\begin{equation}\label{eqn:Kadj_on_DA}
	K^*h' = h \quad \text{for all } h\in \domL.
\end{equation}
Moreover, utilizing the fundamental theorem of calculus and the IbPF, it holds $Kh''~=~-h'$ for all $h\in \ho \cap \Cinfzo$.

\begin{lemma}
	Let $h\in\domL$. Then it holds
	\begin{equation}\label{eqn:wnspace_ltwo_translation}
		-(h'', \bb(\omega))_{\ltwo}=\langle h', \omega\rangle \quad\text{for $\mu$-a.e. $\omega\in\SchwartzFcts$},
	\end{equation}
	where we extended $h'$ to $\R$ by zero to ensure well-definedness of the dual pairing.
\end{lemma}
\begin{proof}
	Let $h\in \ho \cap \Cinfzo$, which implies that $Kh''~=~-h'$, $h'' \in \Czo$ and hence $-(h'', \bb(\omega))_{\ltwo}=\langle h', \omega\rangle$ for $\mu$-a.e. $\omega\in\SchwartzFcts$. In order to prove this fact for all $h\in \ho \cap H^2$, we consider that $h\in \ho \cap \Cinfzo$ is dense in $\ho \cap H^2$ w.r.t.\ the $H^{2}$-norm, which is a consequence of the results in [Adams], Section 6. This provides that for any $h\in \ho \cap H^{2}$ we can find a sequence $(h_n) \subset \ho \cap \Cinfzo$ s.th.\ $h_n \to h$ in $H^2$ for $n \to \infty$, in particular $h_n' \to h'$ and $h_n'' \to h''$ in $\ltwo$ as $n\to\infty$. Interpreting these functions as elements in $L^2(\R)$ by extension to $\R$ by zero, the isometry \eqref{eqn:wna_isometry} yields that $\langle h_n', \cdot \rangle \to \langle h', \cdot \rangle$ in $L^2(\mu)$, and due to the Riesz--Fischer theorem (and dropping to representatives in $\mathcal{L}^2(\mu)$), there exists a subsequence $(h_{n_k})_k$ s.th.\ $\langle h_{n_k}', \omega \rangle \to \langle h', \omega \rangle$ for $\mu$-almost all $\omega \in \SchwartzFcts$. Additionally, it holds $-(h_{n_k}'', \bb(\omega))_{\ltwo} \to -(h'', \bb(\omega))_{\ltwo}$ for all $\omega \in \SchwartzFcts$, which yields the claim.
\end{proof}

\begin{lemma}\label{lem:CMF}
	(Version of the Cameron--Martin Formula) Let $h \in \domL$. Then, the measure $T_{+h}(\mubb)$ is absolutely continuous w.r.t.\ $\mubb$, and its Radon--Nikodym density is given by
	\begin{equation}\label{lem:CMF_formula}
		\frac{dT_{+h}(\mubb)}{d\mubb} = \exp\left( -(h'', \cdot )_{\ltwo} - \frac{1}{2} \| h' \|_{\ltwo}^2 \right) \eqqcolon \xi.
	\end{equation}
\end{lemma}
\begin{proof}
	First note that two probability measures on $(\ltwo,\borel(\ltwo))$ are equal if their respective characteristic functionals agree on $\Czo$. This follows from \cite{VTC87}, p.\ 201, Corollary 1(b), using that $\ltwo$ is Polish and hence Radon, the fact that all probability measures on Radon spaces are Radon, and that $\Czo$ separates points of $\ltwo$, i.e.\ for all $f$, $g\in\ltwo$ with $f\neq g$ there exists a $\eta\in \Czo$ s.th.\ $(\eta, f)_{\ltwo}\neq (\eta,g)_{\ltwo}$. The latter is due to the Hahn--Banach theorem, the Riesz representation theorem and density of $\Czo$ in $\ltwo$. It remains to verify this condition.
	
	Clearly $T_{+h}(\mubb)$ is a probability measure on $(\ltwo,\borel(\ltwo))$. Further, $\xi \mubb$ is a probability measure on $(\ltwo,\borel(\ltwo))$ as well, since $\xi \in L_+^1(\ltwo, \mubb)$ and $\|\xi\|_{L^1}=1$, where the latter follows from a straightforward calculation, using \eqref{eqn:wnspace_ltwo_translation} and \eqref{eqn:wna_int_exp}. Moreover, let $\eta \in \Czo$, then
	\begin{align*}
		\widehat{T_{+h}(\mubb)}\left( \eta \right)
		&= \int_{\ltwo} \exp\left(i (\eta, z)_{\ltwo} \right) dT_{+h}(\mubb)(z)\\
		&= \int_{\SchwartzDistribs} \exp\left( i (\eta, \bb(\omega)+h)_{\ltwo} \right) d\mu(\omega)\\
		&= \int_{\SchwartzDistribs} \exp \left( i \langle K\eta, \omega + h' \rangle \right) d\mu(\omega)\\
		&= \int_{\SchwartzDistribs} \exp \left( i \langle K\eta, \omega \rangle \right) \exp \left( \langle h', \omega \rangle - \frac{1}{2} \| h' \|_{L^2(\R)}^2 \right) d\mu(\omega),
	\end{align*}
	where the penultimate equality holds due to \eqref{eqn:L2WNA_translation_formula} and \eqref{eqn:Kadj_on_DA}, where we extend $K\eta$ by zero, such that we can interpret $K\eta \in L^2(\R)$, and the last equality follows by applying the Cameron--Martin formula from white noise analysis, see Proposition \ref{prop:cmf_wna}. Therefore, note that $K\eta$ can also get interpreted as a restriction of an element in $\SchwartzFcts$, see \cite{GV18}, p.\ 356. We finally conclude that
	\begin{align*}
		\widehat{T_{+h}(\mubb)}\left( \eta \right)
		&= \int_{\SchwartzDistribs} \exp \left( i (\eta, \bb(\omega))_{\ltwo} \right) \exp \left( -(h'', \bb(\omega))_{\ltwo} - \frac{1}{2} \| h' \|_{\ltwo}^2 \right) d\mu(\omega)\\
		&= \widehat{\xi\mubb}\left( \eta \right),
	\end{align*}
	where the first equation holds due to \eqref{eqn:L2WNA_translation_formula} and \eqref{eqn:wnspace_ltwo_translation}. This finishes the proof.
\end{proof}

\subsubsection{Derivation of the Integration by Parts Formulas}

Since we first proof smoothed versions of the targeted IbPFs, we define a smoothed version of the considered indicator function via a Gaussian mollifier for any $\varepsilon>0$ by
\begin{equation*}
	\inds \coloneqq \gmoll \ast \indzi, \quad \text{where } \gmoll (x) \coloneqq \frac{1}{\sqrt{2\pi\varepsilon}} \exp \left(-\frac{x^2}{2\varepsilon}\right) \text{ for all } x\in\R.
\end{equation*}
Consequently, we define the densities $\densaeps,\, \denseps: \ltwo \to \R$ as in \eqref{def:densa} when exchanging $\indzi$ for $\inds$. It analogously holds $0\leq\densaeps,\, \denseps\leq 1$, and with the same arguments as in Lemma \ref{lem:L1+_RayHamza} we conclude that $\densaeps,\, \denseps \in L_+^1(\ltwo, \mubb)$.

\begin{lemma}\label{lem:smoothedIbPF}
	(Smoothed Versions of the IbPFs) For all $\F \in C_b^1(\ltwo)$ and $h \in \ho \cap H^2$ we have the IbPF
	\begin{align}\label{smoothedIbPF}
		\E_{\densaeps\mubb}\left[ \partial_h \F \right] + \E_{\densaeps\mubb}\left[ \F \, (h'',\cdot)_{\ltwo} \right]
		&= - \int_{\SchwartzDistribs} \F(\bb(\omega)) \, \bar{h}_a \gmoll(\ddexta(\bb(\omega))) \, d\mu(\omega)\\
		&= - \llangle \F(\bb), \bar{h}_a \gmoll(\bb_a) \rrangle,\nonumber
	\end{align}
	where $\llangle\cdot,\cdot\rrangle$ denotes the Hida dual pairing, see Subsection \ref{sect:wna}, which extends the canonical scalar product on $L^2(\mu)$. Further, the following IbPF holds:
	\begin{align}\label{smoothedIbPF_rho}
		\E_{\denseps\mubb}\left[ \partial_h \F \right]
		+ \E_{\denseps\mubb}\left[ \F \, (h'', \cdot)_{\ltwo} \right]
		&= - \int_{\SchwartzDistribs} \F(\bb(\omega)) \, \int_{\unitint} \bar{h}_x \gmoll(\ddext(\bb(\omega))) \,dx \, d\mu(\omega)\\
		&= - \bllangle \F(\bb), \int_{(0,1)} \bar{h}_x \gmoll(\bb_x) \, dx \brrangle. \nonumber
	\end{align}
\end{lemma}
\begin{proof}
	A straightforward application of the differentiation lemma for integrals, followed by an application of the Cameron--Martin formula from Lemma \ref{lem:CMF}, yields
	\begin{align*}
		&\int_{\ltwo} \partial_h \F(z) \, d(\densaeps\mubb)(z)\\
		&= \int_{\ltwo} \frac{d}{d\lambda} \F(z+\lambda h) \bigg\rvert_{\lambda=0} \densaeps(z) \, d\mubb(z)\\
		&= \frac{d}{d\lambda} \left[ \int_{\ltwo} \F(z+\lambda h) \, \densaeps((z+\lambda h) -\lambda h) \,d\mubb(z)\right]\bigg\rvert_{\lambda=0}\\
		&= \frac{d}{d\lambda} \left[ \int_{\ltwo} \F(z) \, \densaeps(z-\lambda h)\exp \left( - \lambda (h'', z)_{\ltwo} - \frac{\lambda^2}{2} \| h' \|_{\ltwo}^2 \right) d\mubb\right]\bigg\rvert_{\lambda=0}.
	\end{align*}
	Wanting to apply the differentiation lemma again, we proceed by checking its assumptions: Observe that for any fixed $\lambda\in B_1(0)$, the map
	\begin{equation*}
		\ltwo \incl z \mapsto \F(z) \, \densaeps(z - \lambda h) \exp \left( - \lambda(h'', z)_{\ltwo} - \frac{\lambda^2}{2} \| h' \|_{\ltwo}^2 \right) \in\R
	\end{equation*}
	is in $\mathcal{L}^1(\ltwo,\mubb)$, since $\exp(-\lambda (h'',\cdot)_{\ltwo})\in\mathcal{L}^1(\ltwo,\mubb)$ due to \eqref{eqn:wnspace_ltwo_translation} and \eqref{eqn:wna_int_exp}, and since all other terms are bounded. Further, if $z\in\coec$, then $z+\lambda h\in\coec$ for all $\lambda\in B_1(0)$, since $h\in\coec$. In this case, we can calculate
	\begin{align}\label{eqn:ibp_deriv}
		&\frac{d}{d\lambda} \left[ \densaeps(z - \lambda h)\exp \left( - \lambda (h'',z)_{\ltwo} - \frac{\lambda^2}{2} \| h' \|_{\ltwo}^2 \right) \right]\nonumber\\
		&= - \bar{h}_a \gmoll(\bar{z}_a-\lambda \bar{h}_a) \exp \left( - \lambda(h'',z)_{\ltwo} - \frac{\lambda^2}{2} \| h' \|_{\ltwo}^2 \right)\nonumber\\
		&\quad+\densaeps(z - \lambda h)\left( -(h'',z)_{\ltwo} - \lambda \| h' \|_{\ltwo}^2 \right)\exp \left( - (h'',z)_{\ltwo} - \frac{\lambda^2}{2} \| h' \|_{\ltwo}^2 \right),
	\end{align}
	hence for $\lambda \in B_1(0)$ it holds
	\begin{align*}
		&\left|\frac{d}{d\lambda} \left[ \F(z) \, \densaeps(z - \lambda h)\exp \left( - \lambda(h'',z)_{\ltwo} - \frac{\lambda^2}{2} \| h' \|_{\ltwo}^2 \right) \right] \right|\\
		&\leq |\F(z)| |\bar{h}_a| |\gmoll(\bar{z}_a-\lambda \bar{h}_a)| \exp\Big(|\lambda| |(h'',z)_{\ltwo}|\Big)\exp\left( -\frac{\lambda^2}{2}\|h'\|_{\ltwo}^2 \right)\\
		&\quad + |\F(z)| |\densaeps(z - \lambda h)|  \Big(|(h'',\cdot)_{\ltwo}|+|\lambda|\|h'\|_{\ltwo}^2\Big) \exp\Big(|(h'',z)_{\ltwo}|\Big)\exp\Big( -\frac{\lambda^2}{2}\|h'\|_{\ltwo}^2 \Big)\\
		&\leq \|\F\|_{\infty}\left(\frac{1}{\sqrt{2\pi\varepsilon}}\| h \|_{\infty} + |(h'', z)_{\ltwo} | + \| h' \|_{\ltwo}^2 \right) \exp \Big( |(h'', z)_{\ltwo} | \Big)\\
		&\leq C \, (1+|(h'',z)_{\ltwo}|) \exp \Big( |(h'', z)_{\ltwo} | \Big)
	\end{align*}
	for some constant $C<\infty$. Straightforward calculations, using \eqref{eqn:wnspace_ltwo_translation} and \eqref{eqn:wna_image_meas_lem}, provide that the above upper bound is indeed $\mubb$-integrable. Further, note that the case $z\in\coec^c$ does not need to be examined, since $\coec^c$ is a $\mubb$-null set. Hence the differentiation lemma for integrals and setting $\lambda=0$ in \eqref{eqn:ibp_deriv} implies that
	\begin{align*}
		&\int_{\ltwo} \partial_h \F(z) \, d(\densaeps\mubb)(z)\\
		&= \int_{\ltwo} \F(z) \, \frac{d}{d\lambda} \left[ \densaeps(z - \lambda h)\exp \left( - \lambda(h'', z)_{\ltwo} - \frac{\lambda^2}{2} \| h' \|_{\ltwo}^2 \right) \right]\bigg\rvert_{\lambda=0} d\mubb\\
		&= \int_{\ltwo} \F(z) \Big(-(h'',z)_{\ltwo} \, \densaeps(z) - \bar{h}_a \gmoll(\ddexta(z)) \Big) d\mubb,
	\end{align*}
	which proves \eqref{smoothedIbPF}. Further, the IbPF \eqref{smoothedIbPF_rho} can be derived from the IbPF \eqref{smoothedIbPF} by integration, more precisely
	\begin{align*}
		&\int_{\ltwo} \partial_h \F \, d(\denseps\mubb) + \int_{\ltwo} \F \, (h'',\cdot)_{\ltwo} \, d(\denseps\mubb)\\
		&= \int_{(0,1)} \left[ \int_{\ltwo} \partial_h \F \, d(\rho^{x,\varepsilon}\mubb) + \int_{\ltwo} \F \, ( h'', \cdot)_{\ltwo} \, d(\rho^{x,\varepsilon}\mubb) \right] dx\\
		&= - \int_{(0,1)} \llangle \F(\bb), \bar{h}_x \gmoll(\bb_x) \rrangle \, dx\\
		&= - \bllangle \F(\bb), \int_{(0,1)} \bar{h}_x \gmoll(\bb_x) \, dx \brrangle,
	\end{align*}
	where the first equality follows due to Fubini's theorem, the second equality holds by applying \eqref{smoothedIbPF}, and the third equality follows, since Bochner integrals commute with bounded and linear maps, considering that $\int_{(0,1)} \bar{h}_x \gmoll(\ddext) dx \in L^2(\ltwo, \mubb)$ is indeed well-defined as a Bochner integral, see Lemma \ref{lem:int_onefold_moll}. This finishes the proof.
\end{proof}

The intended IbPFs now follow by taking the limit $\varepsilon\downarrow 0$.

\begin{remark}
	For $h\in\domL$ and $\F\in\exp(C^{\infty})$, both expressions
	\begin{equation*}\label{eqn:ops_before_ext}
		\llangle \F(\bb), \bar{h}_a \DDa \rrangle \quad \text{and} \quad \bllangle \F(\bb), \int_{(0,1)}\bar{h}_x \DDx \, dx \brrangle
	\end{equation*}
	are well-defined. This follows due to $\F(\bb) \in \HidaFcts$ for all $\F\in\exp(C^{\infty})$, see \cite{GV18}, p.\ 356, and Remark 3.3, together with $\bar{h}_a \DDa, \, \int_{(0,1)}\bar{h}_x \DDx \, dx \in \HidaDistribs$, see Subsection \ref{sect:DonskersDelta} and Lemma \ref{lem:BInt_DDx_in_Sprime}, respectively.
\end{remark}

\begin{theorem}[Operator Extension and IbPFs]\label{thm:ibpf}
	Let $h\in\domL$ and $\F \in \exp(C^{\infty})$. Then we have the IbPFs
	\begin{align}
		\E_{\densa\mubb}\left[ \partial_h \F \right] + \E_{\densa\mubb}\left[ \F \, (h'',\cdot)_{\ltwo} \right] &= - \llangle \F(\bb), \bar{h}_a \DDa \rrangle\label{IbPF:rho_x}\\
		\text{and}\quad\E_{\dens\mubb}\left[ \partial_h \F \right] + \E_{\dens\mubb}\left[ \F \, (h'',\cdot)_{\ltwo} \right] &= - \bllangle \F(\bb), \int_{(0,1)}\bar{h}_x \DDx \, dx \brrangle\label{IbPF:rho}.
	\end{align}
	Further, the linear operators
	\begin{align*}
		\exp(C^{\infty})\incl \F &\mapsto \llangle \F(\bb), \bar{h}_a \DDa \rrangle \in\R\\
		\text{and}\quad\exp(C^{\infty})\incl \F &\mapsto \bllangle \F(\bb), \int_{(0,1)}\bar{h}_x \DDx \, dx \brrangle \in\R
	\end{align*}
	extend uniquely to $W^{1,2}(\ltwo, \mubb)$. Identifying these operators with their extensions in $W^{-1,2}(\ltwo, \mubb)$, \eqref{IbPF:rho_x} and \eqref{IbPF:rho} hold for all $\F \in C_b^1(\ltwo)$.
\end{theorem}
\begin{proof}
	For all $p \in [1,\infty)$, it holds that $\densaeps$, $\densa \in L^p(\ltwo, \mubb)$, since both maps are measurable and bounded. Moreover, we have $\densaeps \to \densa$ in $L^p(\ltwo, \mubb)$ for $\varepsilon \downarrow 0$ due to the dominated convergence theorem, considering that an element $\psi\in\ltwo$ is $\mubb$-a.s.\ contained in $\coec$ and satisfies $\bar{\psi}_a \neq 0$, and for such $\psi$ it holds
	\begin{equation*}
		\densaeps(\psi) = \int_{-\infty}^{\bar{\psi}_a} \gmoll(x) \, dx
		\; \xrightarrow{\varepsilon\downarrow 0} \;
		\densa(\psi)
		=\begin{cases}
			1, &\bar{\psi}_a> 0,\\0, &\bar{\psi}_a < 0.
		\end{cases}
	\end{equation*}
	Hence, considering that $\partial_h \F,\, \F(\cdot) (h'',\cdot)_{\ltwo}\in L^2(\ltwo, \mubb)$, an application of the Cauchy--Schwartz inequality yields
	\begin{equation*}
		\left|\int_{\ltwo} \left( \partial_h \F+\F\,(h'',\cdot)_{\ltwo} \right)(\densaeps - \densa) \, d\mubb \right|
		\leq \| \partial_h \F+ \F \, (h'',\cdot)_{\ltwo} \|_{L^2(\ltwo)} \|\densaeps-\densa\|_{L^2(\ltwo)}
	\end{equation*}
	where the upper bound goes to zero as $\varepsilon\downarrow 0$, implying that the left hand side (LHS) of Equation \eqref{smoothedIbPF} converges to the LHS of Equation \eqref{IbPF:rho_x} as $\varepsilon\downarrow 0$.
	If we assume that $\F \in \exp(C^{\infty})$, the same holds for the respective RHSs, considering that in this case we have $\F(\bb)\in\HidaFcts$ and applying that $\gmoll(\bb_a) \to \DDa$ in $\HidaDistribs$ as $\varepsilon\downarrow 0$, see Subsection \ref{sect:DonskersDelta}. This provides that the IbPF \eqref{IbPF:rho_x} holds for all $\F \in \exp(C^{\infty})$ and $h \in \domL$.
	
	Since $\exp(C^{\infty})$ is dense in $W^{1,2}(\ltwo,\mubb)$ and
	\begin{align*}
		&|\llangle \F(\bb), \bar{h}_a\DDa \rrangle|
		= \left|\int_{\ltwo} \left( \partial_h \F+\F \, (h'',\cdot)_{\ltwo} \right)\densa d\mubb \right|\\
		&\leq \| \partial_h \F \|_{L^2(\ltwo)} + \| (h'',\cdot)_{\ltwo} \|_{L^2(\ltwo)} \| \F \|_{L^2(\ltwo)}\\
		&\leq \| h \|_{\ltwo} \| D \F \|_{L^2(\ltwo;\ltwo)} + \|h'\|_{\ltwo} \| \F \|_{L^2(\ltwo)}
		\leq 2 \max\{\| h \|_{\ltwo}, \|h'\|_{\ltwo}\} \, \| \F \|_{W^{1,2}(\ltwo, \mubb)}
	\end{align*}
	for all $\F \in \exp(C^{\infty})$, the linear operator $\exp(C^{\infty})\incl \F \mapsto \llangle \F(\bb), \bar{h}_a \DDa\rrangle \in\R$ uniquely extends to $W^{1,2}(\ltwo,\mubb)$ via the BLT theorem. Here, the first inequality follows from $|\densa|\leq 1$ and the Cauchy--Schwarz inequality.
	
	With further investigation, we find the approximation
	\begin{equation*}
		\llangle (\cdot)\circ \bb, \bar{h}_a \gmoll(\bb_a)\rrangle \xrightarrow{\varepsilon\downarrow 0} \llangle (\cdot)\circ \bb, \bar{h}_a \DDa\rrangle \;\; \text{in $W^{-1,2}(\ltwo,\mubb)$},
	\end{equation*}
	which can be proven as follows: Let $\mathcal{D}$ be the closed unit ball in $\exp(C^{\infty})$ equipped with the $W^{1,2}(\ltwo, \mubb)$-norm, which is a dense subset of the closed unit ball in $W^{1,2}(\ltwo, \mubb)$, on which the IbPF \eqref{IbPF:rho_x} is known to hold. Therefore we can estimate
	\begin{align*}
		&\| \llangle (\cdot)\circ \bb, \bar{h}_a \gmoll(\bb_a)-\bar{h}_a \DDa\rrangle \|_{W^{-1,2}(\ltwo, \mubb)}\\
		&= \sup_{\F \in \mathcal{D}} |\llangle \F(\bb), \bar{h}_a \gmoll(\bb_a)-\bar{h}_a \DDa \rrangle| \\
		&= \sup_{\F \in \mathcal{D}} \left|\int_{\ltwo} \left( \partial_h \F+\F \, (h'',\cdot)_{\ltwo} \right)(\densaeps - \densa) d\mubb \right| \\
		&\leq \sup_{\F \in \mathcal{D}} \big( \| h \|_{L^2} \| D\F \|_{L^2} \| \densaeps -\densa \|_{L^2} + \| \F\|_{L^2} \| (h'',\cdot)_{\ltwo} \|_{L^3} \| \densaeps -\densa \|_{L^6} \big)\\
		&\leq \big(\| \densaeps -\densa \|_{L^2} + \| \densaeps -\densa \|_{L^6}\big) \big(\| h \|_{L^2} + \| (h'',\cdot)_{\ltwo} \|_{L^3}\big) \, C' \sup_{\F \in \mathcal{D}} \| \F \|_{W^{1,2}(\ltwo)}
	\end{align*}
	for some constant $C'<\infty$, where the first estimate follows by applying H\"{o}lder's inequality twice. Note that $\| (h'',\cdot)_{\ltwo} \|_{L^3}$ is indeed finite. Since the last upper bound converges to zero as $\varepsilon\downarrow 0$, we obtain that for all $\F \in C_b^1(\ltwo)$, the RHS of \eqref{smoothedIbPF} converges to the RHS of \eqref{IbPF:rho_x} and thus proves the IbPF \eqref{IbPF:rho_x}.
	
	In order to show the IbPF \eqref{IbPF:rho}, we proceed in a similar manner. First, we show that for any $p \in [1,\infty)$ it holds $\denseps \to \dens$ in $L^p(\ltwo, \mubb)$ as $\varepsilon \downarrow 0$, indeed: Aiming to apply the dominated convergence theorem and considering that $\denseps$, $\dens$ are measurable, non-negative and bounded by $1 \in L^p(\ltwo,\mubb)$, it remains to check that $\denseps \to \dens$ $\mubb$-a.s.\ as $\varepsilon \downarrow 0$. Therefore, we choose $\psi \in \coec$ s.th.\ $dx\{\psi = 0\}=0$, which is $\mubb$-a.s.\ the case, see \cite{GV18}, Lemma 5.1. Then for any arbitrary but fixed $S>0$, choose $R>0$ small enough s.th.\ $dx(M_R)<S/2$, where $M_R\coloneqq \{\psi\in B_R(0)\}$, which is possible due to $dx\{ \psi \in B_D(0) \} \to dx\{\psi = 0\}=0$ as $D \downarrow 0$ by the continuity of measures from above. Further, basic arguments provide the existence of some $\delta > 0$ small enough s.th.\ for all $\varepsilon \in (0,\delta)$ it holds
	\begin{equation*}\label{eqn:IndApprox}
		\big| \inds(x) - \indzi(x) \big| < S/2 \quad \text{for all } x \in B_{R}(0)^c.
	\end{equation*}
	These choices for $R$ and $\delta$ yield that for all $\varepsilon \in (0,\delta)$ it holds
	\begin{equation*}
		|\denseps(\psi) - \dens(\psi)|
		\leq \int_{M_R^c} \left|\inds(\bar{\psi}_x) - \indzi(\bar{\psi}_x) \right| dx  + \int_{M_R} 1 \, dx
		< S,
	\end{equation*}
	which proves the claimed almost sure convergence.	
	Considering that $\int_{(0,1)} \bar{h}_x\gmoll(\bb_x) \, dx$ converges to $\int_{(0,1)} \bar{h}_x \DDx \,dx$ in $\HidaDistribs$, see Lemma \ref{lem:conv_int_onefold_moll}, and the fact that $|\dens|\leq 1$, the IbPF \eqref{IbPF:rho} can be shown by tracing the arguments above for showing the IbPF \eqref{IbPF:rho_x}, replacing $\densa$ and $\densaeps$ by $\dens$ and $\denseps$, respectively.
\end{proof}

\section{Approximation by Vector Measures}
With the aim of representing the RHSs of the IbPFs \eqref{IbPF:rho_x} and \eqref{IbPF:rho} as integral operators w.r.t.\ vector measures of bounded variation, we begin with an introduction to vector measures and present some key theorems in this context.

\subsection{A Brief Introduction of Vector Measures}
In this introduction, we largely follow \cite{MS94}, Section 1. As a standard reference for countably additive set functions in general, \cites{DS88, DU77} is recommended.

Let $\mathcal{F}$ be a $\sigma$-algebra on a set $X$ and let $(B, \| \cdot \|_B)$ denote a real Banach space. Then, $\sigma$-additive functions $V: \mathcal{F} \to B$ are called \emph{vector measures}, and we denote the set of all vector measures by $\ca(\mathcal{F}, B)$. Note that $\ca(\mathcal{F}, \R)$ is exactly the set of all finite signed measures as defined in standard literature, hence the term ``vector measures''. The \emph{variation} of a vector measure $V \in \ca(\mathcal{F},B)$ is the measure (indeed) $\var{V}: \mathcal{F} \to [0,\infty]$ defined by
\begin{equation*}
	\var{V}(E) \coloneqq \sup \bigg\{ \sum_n \|V(E_n)\|_B \, \bigg| \, \text{$(E_n) \subseteq \mathcal{F}$ is a finite partition of $E$} \bigg\}
\end{equation*}
for all $E\in\mathcal{F}$. We define, with abuse of notation, that $\var{V} \coloneqq \var{V}(X)$, and say that $V$ is of \emph{bounded variation} iff $\var{V} < \infty$. Further, the \emph{semivariation} $\svar{V}: \mathcal{F} \to [0,\infty]$ of $V$ is defined by
\begin{equation*}
	\svar{V}(E) \coloneqq \sup \big\{ \var{\varphi(V)}(E) \,\big|\, \varphi \in B', \|\varphi\|_{B'} \leq 1 \big\}
\end{equation*}
for all $E\in\mathcal{F}$, where $B'$ denotes the topological dual space of $B$. Let us write, with abuse of notation, that $\svar{V} \coloneqq \svar{V}(X)$. If not specified otherwise, we say that a set $\mathcal{V} \subseteq \ca(\mathcal{F}, B)$ is \emph{uniformly bounded} iff it is uniformly bounded w.r.t.\ $\svar{\cdot}$. We further have $\| V(\cdot) \|_B \leq \svar{V}(\cdot) \leq \var{V}(\cdot)$.

For a $\mathcal{F}$-measurable simple function $f = \sum a_n \ind_{E_n} : X \to \R$\ (where we assume that $(E_n)\subseteq \mathcal{F}$ is a finite partition of $X$) and $V \in \ca(\mathcal{F},B)$ we define the linear integral
\begin{equation*}
	\int_X f dV \coloneqq \sum a_n V(E_n), \text{ satisfying } \Big\| \int_X f dV \Big\|_B \leq \| f \|_{\infty} \svar{V}
\end{equation*}
Using the completeness of $B$, we can canonically extend this integral for any uniform limit $f:X \to \R$ of a sequence $(f_n)$ of $\mathcal{F}$-measurable simple functions, in particular every bounded $\mathcal{F}/\borel(\R)$-measurable function $f:X\to\R$, and if $\mathcal{F}$ is a Borel $\sigma$-algebra, this applies to all $f \in C_b(X)$. As usual, we set $\int_A f dV \coloneqq \int_X \ind_A \cdot f \, dV$ for every $A\in\mathcal{F}$. Moreover, for any $\varphi \in B'$ and any function $f:X \to \R$ which is a uniform limit of a sequence of $\mathcal{F}$-measurable simple functions we have $\varphi\left( \int_X f dV \right) = \int_X f d\left( \varphi(V) \right)$.

Now suppose that $X$ is a topological space, then a family $\mathcal{V} \subseteq \ca(\borel(X), B)$ is \emph{uniformly tight (w.r.t.\ $\svar{\cdot}$)} iff
\begin{equation}\label{def:UnifTight}
	\forall \varepsilon > 0 \, \exists K \subseteq X \, \text{compact} \; \forall V \in \mathcal{V}: \svar{V}(X \setminus K) < \varepsilon.
\end{equation}
Accordingly, we call a family $\mathcal{V} \subseteq \ca(\borel(X), B)$ \emph{uniformly tight w.r.t.\ $\var{\cdot}$} iff \eqref{def:UnifTight} holds true even when $\svar{\cdot}$ is interchanged for $\var{\cdot}$. Further, a sequence of vector measures $(V_n)$ in $\ca(\borel(X), B)$ \emph{converges weakly to} $V \in \ca(\borel(X), B)$ iff
\begin{equation*}
	\forall f \in C_b(X), \varphi \in B': \; \varphi \left( \int_X f dV_n \right) \xrightarrow{n \to \infty} \varphi \left( \int_X f dV \right).
\end{equation*}
A family $\mathcal{V} \subseteq \ca(\borel(X), B)$ is called \emph{weakly sequentially compact} iff every sequence $(V_n)$ in $\mathcal{V}$ has a subsequence that converges weakly to some $V \in \ca(\borel(X),B)$.

The following theorem is central for Section \ref{sect:approx_sequ_and_weak_limits}.

\begin{theorem}\label{thm:GenProkhorov}
	(Generalized Version of Prokhorov's Theorem) Let X be a metric space with Borel $\sigma$-field $\borel(X)$ and let $B$ be a reflexive Banach space. Then any uniformly bounded and uniformly tight (both w.r.t.\ $\svar{\cdot}$) family $\mathcal{V} \subseteq \ca(\borel(X),B)$ is weakly sequentially compact.
\end{theorem}
\begin{proof}
	See \cite{MS94}, Corollary 1.6.
\end{proof}

We also use that if a sequence $(V_n)$ in $\ca(\borel(X), B)$, where $X$ is a metric space and $B$ a Banach space, converges weakly to $V$ and is uniformly bounded w.r.t.\ $\var{\cdot}$, then bounded variation is preserved in the weak limit, i.e.\ $V$ is of bounded variation, see Lemma \ref{lem:vecmeas_variationformula}.

\subsection{Approximating Sequences of Vector Measures and their Weak Limits}\label{sect:approx_sequ_and_weak_limits}
With the tools developed in the previous subsection, we can now address the representation problem related to the RHS of \eqref{IbPF:rho_x}. Recalling
\begin{equation*}
	L_+^2(\SchwartzDistribs, \mu) \incl \gmoll(\bb_a) \xrightarrow[\varepsilon \downarrow 0]{\HidaDistribs} \DDa \in \HidaDistribs,
\end{equation*}
see Subsection \ref{sect:DonskersDelta}, we define a finite measure $\approxmeasx$ on $(\ltwo, \borel(\ltwo))$ by 
\begin{equation*}
	\approxmeasx(B) \coloneqq \int_{B\cap \coec} \gmoll(\bar{z}_a) \, d\mubb(z) \quad \text{for all } B \in \borel(\ltwo).
\end{equation*}
Further, for all $\varepsilon>0$, we define an approximating vector measure $\vvmeasepsx \in \ca(\borel(\ltwo),\hno)$ via
\begin{equation*}\label{def:vvmeasepsx}
	\vvmeasepsx(h,B) \coloneqq \int_{B\cap \coec} \bar{h}_a \gmoll(\bar{z}_a) \, d\mubb(z) = \approxmeasx(B) \, \delta_a(h)
\end{equation*}
for all $h\in\ho$ and $B\in\borel(\ltwo)$. Note that $\vvmeasepsx$ is indeed a well-defined element of $\ca(\borel(\ltwo),\hno)$, since $\vvmeasepsx(\cdot,B) = \approxmeasx(B) \, \delta_a \in \hno$ for all $B\in\borel(\ltwo)$, and $\vvmeasepsx(h,\cdot)$ is $\sigma$-additive for all $h\in\ho$ due to $\sigma$-additivity of $\approxmeasx$.

\begin{lemma}\label{lem:approxmeasx_unif_bounded}
	The family $(\approxmeasx)_{\varepsilon>0}$ of measures on $(\ltwo, \borel(\ltwo))$ is uniformly bounded for any $a\in (0,1)$. More precisely, it holds for all $\varepsilon>0$ that
	\begin{equation*}
		\approxmeasx(\ltwo) \leq (2\pi a (1-a))^{-1/2} < \infty.
	\end{equation*}
\end{lemma}
\begin{proof}
	In use of \eqref{eqn:wna_image_meas_lem} and $\| q_a \|_{L^2(\R)}^2 = a(1-a)$, we calculate
	\begin{align*}
		\approxmeasx(\ltwo)
		&= \int_{B\cap \coec} |\gmoll(\bar{z}_a)| \, d\mubb(z)
		= \int_{\SchwartzDistribs} \gmoll(\langle q_a, \omega \rangle) \, d\mu(\omega)\\
		&= \int_{\R} \gmoll(y) \gamma_{a(1-a)}(y) \, dy
		\leq (2\pi a (1-a))^{-1/2} < \infty. \qedhere
	\end{align*}
\end{proof}

\begin{remark}\label{rem:varmeas_of_vvmeasepsx}
	For any $h\in \ho$, the variation of the signed measure $\vvmeasepsx(h,\cdot) = \bar{h}_a \approxmeasx(\cdot)$ is the (finite) measure $\var{\vvmeasepsx(h, \cdot)}(\cdot) = |\bar{h}_a| \, \approxmeasx(\cdot)$, which implies
	\begin{equation}\label{eqn:vvmeasepsx_var}
		\var{\vvmeasepsx}(\cdot) = \sup_{h\in B_1(0)} |\bar{h}_a| \, \approxmeasx(\cdot) = \|\delta_a\|_{\hno} \, \approxmeasx(\cdot).
	\end{equation}
\end{remark}

\begin{corollary}\label{cor:vvmeasepsx_unif_bounded}
    The family $(\vvmeasepsx)_{\varepsilon >0} \subseteq \ca\left( \borel(\ltwo), H_0^{-1}(0,1) \right)$ is uniformly bounded w.r.t.\ $\var{\cdot}$, and hence uniformly bounded w.r.t.\ $\svar{\cdot}$.
\end{corollary}
\begin{proof}
	This is a direct consequence Remark \ref{rem:varmeas_of_vvmeasepsx} and Lemma \ref{lem:approxmeasx_unif_bounded}.
\end{proof}

Proceeding similarly for $\dens$, we reconsider that for any $h \in \ho$ it holds
\begin{equation*}
    L^2(\SchwartzDistribs, \mu) \incl \int_{(0,1)} \bar{h}_x \gmoll(\bb_x(\omega)) \, dx \xrightarrow[\varepsilon \downarrow 0]{\HidaDistribs} \int_{(0,1)} \bar{h}_x \delta_0(\bb_x) \, dx \in \HidaDistribs,
\end{equation*}
where we interpret all considered integrals as Bochner integrals, see Lemma \ref{lem:conv_int_onefold_moll}, Corollary \ref{cor:BInt_in_L2} and Lemma \ref{lem:BInt_DDx_in_Sprime}. This motivates the definition of the approximating vector measure $\vvmeaseps \in \ca(\borel(\ltwo), H_0^{-1}(0,1))$ by
\begin{align}\label{def:vvmeaseps}
	\vvmeaseps(h,B) \coloneqq \int_{B\cap \coec} \int_{(0,1)} \bar{h}_x\gmoll(\bar{z}_x) \, dx \, d\mubb(z) \quad \text{for all } h\in\ho, \, B\in\borel(\ltwo),
\end{align}
for all $\varepsilon>0$. Here, checking that $\vvmeaseps$ is indeed a well-defined element of $\ca(\borel(\ltwo), H_0^{-1}(0,1))$ is a bit more involved:

\begin{lemma}[Well-Definedness of $\vvmeaseps$]\label{lem:welldef_vvmeaseps}
	For any $\varepsilon>0$, the vector-valued set function $\vvmeaseps$, as defined in \eqref{def:vvmeaseps}, is a well-defined element of $\ca(\borel(\ltwo), H_0^{-1}(0,1))$.
\end{lemma}
\begin{proof}
	The expression $\vvmeaseps(h,B)$ is well-defined for all $h\in \ho$ and $\varepsilon>0$, since the integrand $(0,1)\times\ltwo \incl (x,z) \mapsto \bar{h}_x\densxbv(z) \in \R$ is bounded, measurable and therefore integrable, where measurability is due to continuity and boundedness w.r.t.\ the $x$-component and measurability w.r.t.\ the $z$-component.
	
	Further, let $B\in\borel(\ltwo)$ arbitrary. Then, $\vvmeaseps(\cdot,B)$ is clearly linear on $\ho$, and moreover bounded, considering the estimate
	\begin{equation}\label{eqn:vvmeaseps_modulus_est}
		|\vvmeaseps(h,B)|
		\leq C \|h\|_{\ho} \int_{B \cap\coec} \int_{(0,1)} \gmoll(\bar{z}_x) \,dx \, d\mubb(z)
		\leq C \|h\|_{\ho} \int_{(0,1)} \frac{1}{\sqrt{2\pi x(1-x)}} \, dx,
	\end{equation}
	for some $C<\infty$ independent of $h$ and $B$, where the second inequality holds due to Tonelli's theorem and Lemma \ref{lem:approxmeasx_unif_bounded}. This shows that $\vvmeaseps(\cdot, B) \in \hno$, i.e.\ $\vvmeaseps$ is well-defined as a set function, and we moreover obtain the upper bound
	\begin{equation}\label{eqn:vvmeaseps_norm_est}
		\| \vvmeaseps(\cdot,B) \|_{\hno}
		\leq C \int_{B\cap\coec} \int_{(0,1)} \gmoll(\bar{z}_x) \, dx \, d\mubb(z).
	\end{equation}
	It remains to show countable additivity: Let $(B_n)_{n\in\N}$ a sequence of disjoint measurable sets and let $A_n\coloneqq \bigsqcup_{n=N+1}^{\infty}B_n$ for all $N\in\N$. Then it holds
	\begin{equation*}
		\bigg\| \vvmeaseps\bigg(\cdot, \bigsqcup_{n=0}^{\infty} B_n\bigg) - \sum_{n=0}^{N} \vvmeaseps(\cdot, B_n) \bigg\|_{\hno}
		\leq C \int_{A_n} \int_{(0,1)} \gmoll(\bar{z}_x) \, dx \, d\mubb(z)
		\xrightarrow{N\uparrow \infty} 0,
	\end{equation*}
	where the estimate holds due to \eqref{eqn:vvmeaseps_norm_est} and the convergence follows from the dominated convergence theorem, considering that $\ind_{A_n}\downarrow 0$ as $N\uparrow \infty$. This finalizes the proof.
\end{proof}

\begin{remark}
	[The Variation of $\vvmeaseps$] The upper bound for the norm of $\vvmeaseps(\cdot,B)$, as presented in \eqref{eqn:vvmeaseps_norm_est}, implies that, for $C<\infty$ chosen as in Subsection \ref{sect:misc_intro}, we have
	\begin{equation}\label{eqn:vvmeaseps_var_est}
		\var{\vvmeaseps}(B) \leq C \int_B \int_{(0,1)} \gmoll(\bar{z}_x) \, dx \, d\mubb(z) \quad \text{for all } B \in \borel(\ltwo).
	\end{equation}
\end{remark}

\begin{lemma}\label{lem:nu_unifbounded}
    The family $(\nu_{\varepsilon})_{\varepsilon\in (0,1]}$ is uniformly bounded w.r.t.\ $\var{\cdot}$.
\end{lemma}
\begin{proof}
	Let $(B_n)$ a finite measurable partition of $\ltwo$. Using \eqref{eqn:vvmeaseps_norm_est} and considering the last estimate in \eqref{eqn:vvmeaseps_modulus_est}, we obtain
	\begin{align*}
		\sum_n \|\vvmeaseps(\cdot, B_n)\|_{\hno}
		&\leq C \sum_n \int_{B_n} \int_{(0,1)} \gmoll(\bar{z}_x) \,dx \,d\mubb(z)\\
		&= C \int_{\ltwo} \int_{(0,1)} \gmoll(\bar{z}_x) \, dx \, d\mubb(z) < \infty.
	\end{align*}
	Since the last bound is independent of the choice of $(B_n)$, the claim follows.
\end{proof}

In the following, we introduce a class of compact sets which are useful for analyzing uniform tightness of the considered families of vector-valued measures:
\begin{definition}[Sets of H\"{o}lder Continuous Functions]\label{def:hoelder_sets_def}
    For $n \in \N$, let
    \begin{equation*}
        \hc_n \coloneqq \Big\{ \psi \in C_0[0,1] \,\Big|\, \forall \, x,y \in [0,1]: \; | \psi(x)-\psi(y) | \leq n\, |x-y|^{1/n} \Big\},
    \end{equation*}
    and let $\hc \coloneqq \bigcup_{n>0} \hc_n$.
\end{definition}

\begin{lemma}[Properties of $\hc_n$ and $\hc$]\label{lem:hoelder_set_props}
	The following holds true:
	\begin{enumerate}
		\item For any $n\in\N$, the canonical embedding $\iota (\hc_n)$ is a compact subset of $\ltwo$. From now on, we identify $\hc_n$ with $\iota(\hc_n)$, and similarly for $\hc$.
		\item $\hc_n\in\borel(\ltwo)$ for all $n\in\N$, and $\hc\in\borel(\ltwo)$.
		\item $\hc_n \uparrow \hc$ for $n\to\infty$.
		\item $\hc$ is a $\mubb$-almost sure set.
	\end{enumerate}
\end{lemma}
\begin{proof}
	For a given $n \in \N$, it is straightforward to check that the set $\hc_n$ is both uniformly bounded and uniformly equicontinuous. Therefore, the Arzel\`{a}--Ascoli theorem implies that $\hc_n$ is a relatively compact subset of $\co$ (equipped with the norm $\| \cdot \|_{\infty}$). It is even compact, since it is closed in $\co$: Let $(\psi_m)_{m\in\N}$ a sequence in $\hc_n$ converging to $\psi \in \co$, then
	\begin{align*}
		| \psi(x) - \psi(y) | &\leq |\psi(x) - \psi_m(x)| + |\psi_m(x) - \psi_m(y)| + |\psi_m(y) - \psi(y)| \\
		&\leq 2 \|\psi - \psi_m \|_{\infty} + n \, |x-y|^{1/n}
		\xrightarrow{m \to \infty} n \, |x-y|^{1/n},
	\end{align*}
	hence $\psi \in \hc_n$. Now, continuity of the canonical embedding $\iota: \co \hookrightarrow \ltwo$ yields (i). Since $(\ltwo, \| \cdot \|_{\ltwo})$ is Hausdorff, all compact sets are closed in $\ltwo$, which implies (ii). Further, (iii) holds, since the sequence $(\hc_n)_{n\in\N}$ is a monotonically increasing, i.e.\ $\hc_n \subseteq \hc_{n+1}$ for all $n\in\N$, which holds since $(n \, (\cdot)^{1/n})_{n\in\N}$ is a pointwise monotonically increasing sequence in $C[0,1]$. Finally, (iv) follows from
	\begin{equation*}
		\hc = \{ \psi \in \co \mid \exists \, C<\infty, \alpha >0:\; \psi \text{ H\"{o}lder cont.\ with const.\ $C$ and exponent $\alpha$} \},
	\end{equation*}
	and the fact that the paths of the standard Brownian bridge are $\mubb$-a.s.\ H\"{o}lder continuous, see \cite{GV18}, p.\ 341. This finalizes the proof.
\end{proof}

The following construction in Definition \ref{def:l_and_bb_pin} and the statement of Lemma \ref{lem:independence} adapt the arguments from the proof of \cite{Zam05}, Lemma 3.3, transitioning from the case of the Brownian motion to that of the Brownian bridge.

\begin{definition}\label{def:l_and_bb_pin}
	Let $x \in (0,1)$, and let $l^x \in \co$ be defined by
	\begin{equation*}
		l^x_{\xi} \coloneqq \frac{\xi}{x} \, \ind_{[0,x]}(\xi) + \frac{1-\xi}{1-x} \, \ind_{(x,1]}(\xi), \quad \xi \in [0,1].
	\end{equation*}
	Consequently, we define the stochastic process $\bb^{\pin,x} \coloneqq \big(\bb_{\xi} - \bb_x  l_{\xi}^x \big)_{\xi \in [0,1]}$. Note that the paths of $\bb^{\pin,x}$ are in $\co$, and we have $\bb^{\pin,x}_x = 0$ by construction.
\end{definition}

\begin{lemma}[An Integration Formula]\label{lem:independence}
	Let $x\in (0,1)$. Then it holds that $\bb_x$ and $\bb^{\pin,x}$ (as introduced in Definition \ref{def:l_and_bb_pin}) are independent. Further, the integration formula
    \begin{align}\label{eqn:partial_integral}
    	\Emubb\left[ \F(\bb) \psi (\bb_x) \right]
    	&= \int_{\R} \Emubb\left[ \F(\bb^{\pin,x}+y l^x)\right] \psi (y) \, \mathcal{N}(0,x(1-x))(dy) \nonumber\\
    	&= \int_{\R} \Emubb\left[ \F(\bb+(y-\bb_x) l^x)\right] \psi (y) \, \mathcal{N}(0,x(1-x))(dy)
    \end{align}
    holds for any $\psi \in C_b(\R)$ and any bounded Borel-measurable map $\F \colon \ltwo \to \R$.
\end{lemma}
\begin{proof}
	This follows by a straightforward adaption of \cite{Zam05}, Lemma 3.3.
\end{proof}

The next lemma is crucial for leveraging Lemma \ref{lem:independence} in the upcoming uniform tightness proofs.

\begin{lemma}[Roughness of Pinned Brownian Bridges]\label{lem:roughness_of_pinning}
    Let $k \in \N$, $r\in \R_{>0}$, $x\in (0,1)$, and $m(k,r,x) \in \N$ with $m(k,r,x) \geq k+\left(r+k\right) \left(\min\{x,1-x\}\right)^{-1}$. Then it holds for all $y \in B_r(0)$ that
    \begin{equation*}
        \{z \in \co \mid z+(y-z_x) \, l^x \in \hc_{m(k,r,x)}^c\} \subseteq \hc_k^c.
    \end{equation*}
	Note that the complements above can be seen in $\co$ as well as in $\ltwo$. Qualitatively speaking, this means that if the modified map $z+(y-z_x) \, l^x$ is very rough, then $z$ has to be of a certain roughness, too.
\end{lemma}
\begin{proof}
    We give a proof by contraposition: Let $z \in \hc_k$. Then it holds that\newline $z+(y-z_x)\,l^x \in \hc_{m(k,r,x)}$, since
    \begin{align*}
    	&|(z_{\xi}+(y-z_x)\, l_{\xi}^x)-(z_{\zeta}+(y-z_x)\, l_{\zeta}^x)|
    	= |(z_{\xi}-z_{\zeta})+(y-z_x) (l_{\xi}^x-l_{\zeta}^x)|\\
    	&\leq k|\xi-\zeta|^{1/k} + \left(|y|+|z_x|\right)  \left(\min\{x,1-x\}\right)^{-1}|\xi-\zeta|\\
    	&\leq \big(k+\left(r+k\right) \left(\min\{x,1-x\}\right)^{-1}\big)|\xi-\zeta|^{1/k}\\
    	&\leq m(k,r,x)\,|\xi-\zeta|^{1/m(k,r,x)}
    \end{align*}
	for all $\xi$, $\zeta \in [0,1]$, where the first inequality holds due to Lipschitz continuity of $l^x$ with the Lipschitz constant $(\min\{x,1-x\})^{-1}$. This finishes the proof.
\end{proof}

We are now fully equipped to show uniform tightness of the considered families of (vector) measures.

\begin{lemma}[Uniform Tightness of $(\approxmeasx)_{\varepsilon\in (0,1]}$]\label{lem:approxmeasx_unif_tight}
	Let $x \in (0,1)$. Then the family $(\approxmeasx)_{\varepsilon \in (0,1]}$ of finite measures on $\borel(\ltwo)$ is uniformly tight.
\end{lemma}
\begin{proof}
	Note that it is sufficient to show that
	\begin{equation*}
		\forall \boundtight>0 \; \exists \; m\in\N \; \forall \varepsilon \in (0,1]: \quad \approxmeasx \left( \hc_m^c \right) < \boundtight.
	\end{equation*}
	To this end, we choose some $r\in\R_{>0}$ big enough s.th.\ $\int_{B_r(0)^c} \gmoll (y) \, \mathcal{N}(0,a(1-a))(dy) < \boundtight/2$ for all $\varepsilon \in (0,1]$. This is possible, since for all $\varepsilon>0$, H\"{o}lder's inequality yields
	\begin{multline}\label{eqn:tightness_approxmeasx_integral_est}
		\int_{B_r(0)^c} \gmoll (y) \, \mathcal{N}(0,a(1-a))(dy)
		= \int_{\R} \gmoll (y) \, \ind_{B_r(0)^c}(y) \, \gamma_{a(1-a)}(y) \, dy\\
		\leq \|\gmoll\|_{L^1(\R)} \| \ind_{B_r(0)^c} \, \gamma_{a(1-a)} \|_{\infty}
		= 1\cdot \gamma_{a(1-a)}(r) \xrightarrow[]{r \to \infty} 0,
	\end{multline}
	where the last upper bound is independent of $\varepsilon>0$.
	Subsequently, we choose some $k\in\N$ big enough s.th.\ $\mubb(\hc_k^c)\int_{-r}^r \gmoll (y) \, \mathcal{N}(0,a(1-a))(dy) < \boundtight/2$ for all $\varepsilon \in (0,1]$; this is achievable, since $\mubb(\hc_k^c) \xrightarrow{} 0$ as $k \to \infty$ and, in use of the Cauchy--Schwartz inequality, since $\int_{-r}^r \gmoll (y) \, \mathcal{N}(0,a(1-a))(dy) \leq 1\cdot (2\pi a(1-a))^{-1/2}$, where the bound on the RHS is independent of $\varepsilon\in(0,1]$. Lastly, we define the constant $m \coloneqq m(k,r,a) \in \N$ as seen in Lemma \ref{lem:roughness_of_pinning}. Now we can estimate
	\begin{align*}
		\approxmeasx(\hc_m^c) &= \int_{\ltwo} \ind_{\hc_m^c}(z) \, \gmoll(\ddexta(z)) \, d\mubb(z)\\
		&= \int_{\R} \Emubb\left[ \ind_{\hc_m^c}(\bb+(y-\bb_a) l)\right] \gmoll (y) \,\mathcal{N}(0,a(1-a))(dy)\\
		&\leq \int_{B_r(0)} \Emubb\left[ \ind_{\hc_k^c}(\bb)\right] \gmoll (y)\, \mathcal{N}(0,a(1-a))(dy)\\
		&\hphantom{123456789123}+ \int_{B_r(0)^c} \gmoll (y) \,\mathcal{N}(0,a(1-a))(dy) < \boundtight
	\end{align*}
	where the second equality follows from an application of Lemma \ref{lem:independence}, the subsequent inequality is a consequence of Lemma \ref{lem:roughness_of_pinning}, and the last inequality holds due to the choice of $r$ and $k$. This finishes the proof.
\end{proof}

\begin{corollary}\label{cor:UnifTight_sigma}
    The family $(\sigma_{\varepsilon})_{\varepsilon \in (0,1]} \subseteq \ca(\borel(\ltwo), \hno)$ is uniformly tight w.r.t.\ $\var{\cdot}$.
\end{corollary}
\begin{proof}
	Considering \eqref{eqn:vvmeasepsx_var}, this is a direct consequence of Lemma \ref{lem:approxmeasx_unif_tight}.
\end{proof}

\begin{lemma}\label{lem:nu_uniftight}
    The family of vector-valued measures $(\nu_{\varepsilon})_{\varepsilon \in (0,1]}\subseteq \ca(\borel(\ltwo), \hno)$ is uniformly tight w.r.t.\ $\var{\cdot}$ and hence w.r.t.\ $\svar{\cdot}$.
\end{lemma}
\begin{proof}
    This proof builds on the arguments of the uniform tightness proof for $(\approxmeasx)_{\varepsilon\in (0,1]}$, i.e.\ the proof of Lemma \ref{lem:approxmeasx_unif_tight}: Again, note that it is sufficient to prove that
    \begin{equation*}
        \forall \boundtight>0 \; \exists \, m\in\N \; \forall \varepsilon \in (0,1]: \quad \var{\vvmeaseps} \left( \hc_m^c \right) < \boundtight.
    \end{equation*}
	To this aim, we choose $C<\infty$ as in Subsection \ref{sect:misc_intro} and $\delta\in (0,\frac{1}{2})$ small enough s.th.
	\begin{equation*}
		C\int_{(0,1)\setminus [\delta,1-\delta]} \int_{\R} \gmoll(y) \, \mathcal{N}(0,x(1-x))(dy) \, dx
		\leq C\int_{(0,1)\setminus [\delta,1-\delta]} \frac{1}{\sqrt{2\pi x(1-x)}} \, dx
		< \boundtight/3.
	\end{equation*}
	Subsequently, we choose $r\in\R_{>0}$ big enough s.th.
	\begin{multline*}
		C\int_{[\delta,1-\delta]} \int_{B_r(0)^c} \gmoll(y) \, \mathcal{N}(0,x(1-x))(dy) \, dx
		\leq C (1-2\delta)  \sup_{x\in [\delta, 1-\delta]} \gamma_{x(1-x)}(r)\\
		\leq C (1-2\delta) \, \frac{1}{\sqrt{2\pi \delta(1-\delta)}}\exp\left(-2 r^2 \right) < \boundtight/3,
	\end{multline*}
	where the first inequality holds due to \eqref{eqn:tightness_approxmeasx_integral_est}. Further, we choose $k\in\N$ large enough s.th.
	\begin{align*}
		&C\int_{[\delta,1-\delta]} \int_{B_r(0)} \Emubb\left[ \ind_{\hc_k^c}(\bb) \right] \gmoll(y) \, \mathcal{N}(0,x(1-x))(dy) \, dx\\
		&\leq C\mubb(\hc_k^c) \int_{(0,1)} \frac{1}{\sqrt{2\pi x(1-x)}} \, dx < \boundtight/3.
	\end{align*}
	Lastly, we choose $m\in\N$ s.th.\ $m \geq \sup \left(k+(r+k)(\min\{x,1-x\})^{-1}\right)
	= k+(r+k)\cdot\frac{1}{\delta} < \infty$, where the supremum is taken over all $x\in [\delta, 1-\delta]$. These variable choices lead to the estimate
	\begin{align*}
		\var{\vvmeaseps} \left( \hc_m^c \right)
		&\leq C \int_{\ltwo} \ind_{\hc_m^c}(z) \int_{(0,1)} \gmoll(\ddext(z)) \, dx \, d\mubb(z)\\
		&= C\int_{(0,1)} \Emubb\left[ \ind_{\hc_m^c}(\bb) \, \gmoll(\bb_x) \right]  dx\\
		&= C\int_{(0,1)} \int_{\R} \Emubb\left[ \ind_{\hc_m^c}(\bb+(y-\bb_x) \, l^x)\right] \gmoll(y) \, \mathcal{N}(0,x(1-x))(dy) \, dx\\
		&\leq C\int_{(0,1)\setminus [\delta,1-\delta]} \int_{\R} \gmoll(y) \, \mathcal{N}(0,x(1-x))(dy) \, dx\\
		&\qquad + C\int_{[\delta,1-\delta]} \int_{B_r(0)^c} \gmoll(y) \, \mathcal{N}(0,x(1-x))(dy) \, dx\\
		&\qquad + C\int_{[\delta,1-\delta]} \int_{B_r(0)} \Emubb\left[ \ind_{\hc_k^c}(\bb) \right] \gmoll(y) \, \mathcal{N}(0,x(1-x))(dy) \, dx\\
		&< \boundtight/3 + \boundtight/3 + \boundtight/3 = \boundtight,
	\end{align*}
	where the first line follows from \eqref{eqn:vvmeaseps_var_est}, and the second line follows from an application of Fubini's theorem, justified by the measurability and boundedness of the considered integrand. Further, the third and fourth lines follow from Lemma \ref{lem:independence} and Lemma \ref{lem:roughness_of_pinning}, respectively, and the final line from the preceding variable choices. This establishes the claim.
\end{proof}

\begin{theorem}\label{thm:weak_sequential_compactness}
	The families of (vector) measures $(\approxmeasx)_{\varepsilon\in (0,1]}$, $(\vvmeasepsx)_{\varepsilon >0}$ and $(\nu_{\varepsilon})_{\varepsilon \in (0,1]}$ (and all non-empty subfamilies) are weak sequentially compact.
	
	Further, all weak limits of sequences in $(\approxmeasx)_{\varepsilon\in (0,1]}$, $(\vvmeasepsx)_{\varepsilon >0}$ and $(\nu_{\varepsilon})_{\varepsilon \in (0,1]}$ are bounded by $(2\pi a(1-a))^{-1/2}$, bounded by $(2\pi a(1-a))^{-1/2}$ w.r.t.\ $\var{\cdot}$ and bounded w.r.t.\ $\var{\cdot}$, respectively.
\end{theorem}
\begin{proof}
	The considered families of (vector) measures $(\approxmeasx)_{\varepsilon\in (0,1]}$, $(\vvmeasepsx)_{\varepsilon >0}$ and $(\nu_{\varepsilon})_{\varepsilon \in (0,1]}$ are uniformly bounded, see Lemma \ref{lem:approxmeasx_unif_bounded}, Corollary \ref{cor:vvmeasepsx_unif_bounded} and Lemma \ref{lem:nu_unifbounded}, respectively, and tight, see Lemma \ref{lem:approxmeasx_unif_tight}, Corollary \ref{cor:UnifTight_sigma} and Lemma \ref{lem:nu_uniftight}, respectively. Hence the claim follows by an application of Prokhorov's theorem for finite measures, see e.g.\cite{Kle08}, Theorem 13.29, or its analogue for vector measures, i.e.\ Theorem \ref{thm:GenProkhorov} together with Lemma \ref{lem:vecmeas_variationformula} (ii).
\end{proof}

The previous theorem implies that the sequence $(\mu^{(a,1/n)})_{n\in\N}$ of finite measures on $(\ltwo, \borel(\ltwo))$ admits a weakly convergent subsequence. From now on, we fix such a subsequence $(\mu^{(a,\varepsilon_{n,a})})_{n\in\N}$ and denote its weak limit by $\lmeasx$, i.e.
\begin{equation*}
	\lmeasx \xleftarrow{\;w\;} \mu^{(a,\varepsilon_{n,a})} \subseteq (\mu^{(a,1/n)})_{n\in\N}.
\end{equation*}
Note that $\lmeasx$ is a finite measure on $(\ltwo, \borel(\ltwo))$ which is bounded by $(2\pi a(1-a))^{-1/2}$. Likewise, by the previous theorem, the sequence $(\vvmeas^{1/n})_{n\in\N} \subseteq \ca(\borel(\ltwo),\hno)$ admits a weakly convergent subsequence. We fix such a subsequence $(\vvmeas^{\varepsilon'_n})_{n\in\N}$ and write $\vvmeas\in\ca(\borel(\ltwo),\hno)$ for its weak limit, i.e.\
\begin{equation*}
	\vvmeas \xleftarrow{\;w\;} \vvmeas^{\varepsilon'_n} \subseteq (\vvmeas^{1/n})_{n\in\N}.
\end{equation*}
It holds that $\vvmeas$ is bounded w.r.t.\ $\var{\cdot}$. Note that $\lmeasx$ and $\vvmeas$ might a priori depend on the choice of the corresponding weakly convergent subsequence.

\medskip

In the case of $(\vvmeasepsx)_{\varepsilon >0}$, we proceed a bit differently, in order to obtain a weak limit which can be expressed in terms of $\lmeasx$. Namely, we define the vector measure $\vvmeasx \in \ca(\borel(\ltwo),\hno)$ via
\begin{equation}\label{def:vvmeasx}
	\vvmeasx(h,B) \coloneqq \int_{B\cap \coec} \bar{h}_a \, d\lmeasx(z) = \lmeasx(B) \, \delta_a(h)
\end{equation}
for all $h\in\ho$ and $B\in\borel(\ltwo)$. Here, well-definedness follows by the same arguments as for well-definedness of $\vvmeasepsx$, see \eqref{def:vvmeasepsx} and the subsequent paragraph.

\begin{lemma}\label{lem:vvmeasx_weak_conv}
	The sequence of vector measures $(\sigma^{(a,\varepsilon_{n,a})})_{n\in\N} \subseteq \ca(\borel(\ltwo),\hno)$ converges weakly to $\vvmeasx$, i.e.
	\begin{equation*}
		\vvmeasx \xleftarrow{\;w\;} \sigma^{(a,\varepsilon_{n,a})} \subseteq (\sigma^{(a,1/n)})_{n\in\N}.
	\end{equation*}
	Moreover, for every $h\in\ho$, it holds
	\begin{equation}\label{eqn:sigma_mu_connection}
		\var{\vvmeasx(h,\cdot)}(\cdot) = |\bar{h}_x|\,\lmeasx (\cdot) \quad \text{and} \quad \var{\vvmeasx}(\cdot) = \|\delta_x\|_{\hno} \,\lmeasx(\cdot),
	\end{equation}
	which implies that $\vvmeasx$ is bounded w.r.t.\ $\var{\cdot}$ by $(2\pi a(1-a))^{-1/2}$.
\end{lemma}
\begin{proof}
	Weak convergence follows due to
	\begin{align}\label{eqn:weak_lim_of_vvmeasepsx}
		\Big\langle h, &\int_{\ltwo} F \,d\sigma^{(a,\varepsilon_{n,x})} \Big\rangle
		= \int_{\ltwo} F \,d\sigma^{(a,\varepsilon_{n,x})}(h,\cdot)
		= \bar{h}_a \int_{\ltwo} F \,d\mu^{(a,\varepsilon_{n,x})}\nonumber\\
		&\xrightarrow{n\to\infty} \bar{h}_a \int_{\ltwo} F \,d\lmeasx
		= \Big\langle h, \int_{\ltwo} F \,d\vvmeasx \Big\rangle \quad \text{for all } h\in\ho, \,F\in C_b(\ltwo).
	\end{align}
	The remaining statements follow by arguing as in Remark \ref{rem:varmeas_of_vvmeasepsx}.
\end{proof}

\subsection{A Solution to the Representation Problem and Its Immediate Consequences}

\begin{definition}\label{def:polar_decomp}
	For any $V\in\ca(\borel(\ltwo),\hno)$, there exists a $\var{V}$-measurable function $n_V:\ltwo\to \hno$ with $\|n_V\|_Y=1$ $\var{V}$-a.s.\ such that $V=n_V \var{V}$, see e.g.\ \cite{AMMP10}, Subsection 2.1. We refer to the function $n_V$ as a \emph{unit field}, and we call this decomposition the \emph{polar decomposition} of $V$.
\end{definition}

\begin{remark}
	By \eqref{def:vvmeasx} and Definition \ref{def:polar_decomp}, it holds $\ufd{\vvmeasx} = \hat{\delta}_a \coloneqq \|\delta_a\|_{\hno}^{-1} \, \delta_a \in \hno$ (constant) and $\var{\vvmeasx}=\|\delta_a\|_{\hno} \, \lmeasx$.
\end{remark}

Now, we are ready to prove a central theorem of this paper.
\begin{theorem}[Representations of the RHSs of the IbPFs]\label{lem:gen_ibp}
    For all $l \in \domL$ and $\F \in C_b^1(\ltwo)$ the following two equations hold
    \begin{align}
        \llangle \F(\bb), \bar{l}_a\DDa\rrangle &= \int_{\ltwo} \F \langle l,\ufd{\vvmeasx}\rangle \,d\var{\vvmeasx}\label{gen:rho_x},\\
        \bllangle \F(\bb), \int_{[0,1]} \bar{l}_x\DDx dx \brrangle &= \int_{\ltwo} \F \langle l,\ufd{\vvmeas}\rangle \,d\var{\vvmeas}\label{gen:rho}.
    \end{align}
	Here, the LHSs are interpreted as discussed in Theorem \ref{thm:ibpf}. Moreover, for all elements $l\in \domL$ and $G \in (C_b^1)_{H_0^1\cap H^2}$, the two equations
    \begin{align}
        \int_{\ltwo} D^*G \densa \, d\mubb &= \int_{\ltwo} \langle G,\ufd{\vvmeasx}\rangle \,d\var{\vvmeasx}\label{gen:ibp_x},\\
        \int_{\ltwo} D^*G \dens \, d\mubb &= \int_{\ltwo} \langle G,\ufd{\vvmeas}\rangle \,d\var{\vvmeas}\label{gen:ibp}
    \end{align}
    hold. Using these results, \cite{RZZ12}, Theorem 3.1(iii), directly implies that $\densa,\, \dens \in \mathrm{BV}(\ltwo, \ho)$, $V(\densa)=\var{\vvmeasx}(\ltwo)$ and $V(\dens)=\var{\vvmeas}(\ltwo)$.
\end{theorem}
\begin{proof}
    We first prove \eqref{gen:rho_x} for $l \in \ho$ and $\F \in \exp(C^{\infty}) \subseteq C_b(\ltwo)$: In this case, an application of Lemma \ref{lem:vvmeasx_weak_conv} yields
    \begin{align*}
        \llangle \F(\bb), \bar{l}_a\gamma_{\varepsilon_{n,a}}(\bb_a) \rrangle
        &= \int_{\ltwo} \F \, \bar{l}_a \gamma_{\varepsilon_{n,a}}(\ddexta(\cdot)) \, d\mubb
        = \Big\langle l, \int_{\ltwo} \F\, d\sigma^{(a,\varepsilon_{n,a})} \Big\rangle \\
        &\xrightarrow[n\to\infty]{} \Big\langle l, \int_{\ltwo} \F \,d\vvmeasx \Big\rangle
        = \int_{\ltwo} \F \langle l,\ufd{\vvmeasx}\rangle \,d\var{\vvmeasx},
    \end{align*}
	where the first equality follows by application of the transformation theorem for measures. Simultaneously, it holds that
    \begin{equation*}
    	\llangle \F(\bb), \bar{l}_a\gamma_{\varepsilon_{n,a}}(\bb_a) \rrangle \xrightarrow{n \to \infty} \llangle \F(\bb), \bar{l}_a\DDa\rrangle,
    \end{equation*}
	applying \eqref{eqn:DDforBM_approx}, linearity of the dual pairing and and the fact that $\F(\bb)\in \HidaFcts$, see \cite{GV18}, p.\ 356 and Remark 3.3. Therefore, \eqref{gen:rho_x} holds for all $l\in\ho$ and $\F \in \exp(C^{\infty})$, indeed.
	
	In order to generalize this result, let $\F\in C_b^1(\ltwo)$ and choose a sequence $(\F_n)_n \subset \exp(C^{\infty})$ converging to $\F$ in $W^{1,2}(\ltwo,\mubb)$ as $n\to\infty$, which implies
	\begin{equation*}
		\int_{\ltwo} \F_n \langle l,\ufd{\vvmeasx}\rangle \, d\var{\vvmeasx} = \llangle \F_n(\bb), \bar{l}_a\DDa\rrangle \xrightarrow{n\to\infty} \llangle \F(\bb), \bar{l}_a\DDa\rrangle.
	\end{equation*}
	Further, the map $C_b^1(\ltwo)\incl \Psi \mapsto \int_{\ltwo} \Psi \langle l,\ufd{\vvmeasx}\rangle d\var{\vvmeasx} \in\R$ is continuous w.r.t.\ the $W^{1,2}(\ltwo,\mubb)$-norm for all $l\in \ho\cap H^2$, indeed: First consider that
	\begin{multline*}
		\bigg|\int_{\ltwo} \Psi \langle l, n_{\sigma^{(a, \varepsilon_{n,a})}} \rangle \,d\|\sigma^{(a, \varepsilon_{n,a})}\|\bigg|
		= \bigg|\int_{\ltwo} \Psi \, \bar{l}_a \gamma_{\varepsilon_{n,a}}(\ddexta(\cdot))\, d\mubb\bigg|\\
		= \bigg|\int_{\ltwo} \partial_l \Psi + \Psi (l'',\cdot)_{\ltwo} d(\rho^{a,\varepsilon_{n,a}} \mubb)\bigg|
		\leq D \|\Psi\|_{W^{1,2}(\ltwo,\mubb)},
	\end{multline*}
	for some $D<\infty$ independent of $n\in\N$. The second equation is a consequence of \eqref{smoothedIbPF}, and the existence of such a constant $D$ follows from the argument used in Lemma \ref{lem:smoothedIbPF} together with the bound $|\rho^{a,\varepsilon_{n,a}}|\leq 1$. Then continuity follows from
	\begin{equation*}
		\bigg|\int_{\ltwo} \Psi \langle l, n_{\vvmeasx} \rangle \,d\|\vvmeasx\|\bigg|
		\xleftarrow{n\to\infty}\bigg|\int_{\ltwo} \Psi \langle l, n_{\sigma^{(a, \varepsilon_{n,a})}} \rangle \,d\|\sigma^{(a, \varepsilon_{n,a})}\|\bigg| \leq D\|\Psi\|_{W^{1,2}(\ltwo,\mubb)},
	\end{equation*}
	and directly yields \eqref{gen:rho_x} by
	\begin{equation*}
		\int_{\ltwo} \F_n \langle l,\ufd{\vvmeasx}\rangle \,d\var{\vvmeasx} \xrightarrow{n\to\infty} \int_{\ltwo} \F \langle l,\ufd{\vvmeasx}\rangle \,d\var{\vvmeasx}.
	\end{equation*}
	
	Moreover, for all $l \in \domL$ and $\F \in C_b^1(\ltwo)$ we have
    \begin{multline*}
    	\int_{\ltwo} D^*(\F(z)l) \densa\, d\mubb(z)\\
    	= -\int_{\ltwo} (l, D\F(z))_{\ltwo} \densa\, d\mubb(z) +2\int_{\ltwo} \F(z) (Al,z)_{\ltwo} \densa \,d\mubb(z)\\
    	= -\int_{\ltwo} \partial_l \F(z)\densa \,d\mubb(z) - \int_{\ltwo} \F(z) (l'',z)_{\ltwo} \densa \,d\mubb(z)\\
    	= \llangle \F(\bb), \bar{l}_a\DDa \rrangle
    	= \int_{\ltwo} \F \langle l,\ufd{\vvmeasx}\rangle \,d\var{\vvmeasx},
    \end{multline*}
	where the equalities above can be obtained by applying \eqref{eqn:D*on_class_L2vald_fcts}, the fact that $(l, D\F)_{\ltwo}=\partial_l \F$, IbPF \eqref{IbPF:rho_x} and \eqref{gen:rho_x} in this order. Now, \eqref{gen:ibp_x} follows by linear extension. The analogous statements for $\dens$ are proven in exactly the same way, using Lemma \ref{lem:conv_int_onefold_moll}, IbPF \eqref{IbPF:rho} and \eqref{gen:rho}. This finishes the proof.
\end{proof}

\begin{remark}\label{rem:ExplicitExpressions}
	In use of \cite{RZZ12}, Theorem 3.1(ii), (3.13) and (3.14), one can obtain explicit expressions for the variation measure and the unit field of sufficiently regular vector measures. This may be useful for analyzing $\ufd{\vvmeas}$ and $\var{\vvmeas}$ further.
\end{remark}

\begin{corollary}[Process Decomposition]\label{cor:ProcDecomp}
    There is an $\mathcal{E}^{a}$-exceptional set $S^a \subset \ltwo$ s.th.\ for all $z \in \ltwo\setminus S^a$ under $P_z^a$ there exists an $\mathcal{M}_t^a$-cylindrical Wiener process $W^{z,a}$, s.th.\ the sample paths of the associated distorted Ornstein-Uhlenbeck process $M^{a}$ on $\ltwo$ satisfy for all $l \in D(A) \cap \ho$ that
    \begin{multline}\label{eqn:ProcDecomp}
        \langle l, X_{t}^a-X_0^a\rangle = \int_0^{t} \langle l, dW_s^{z,a}\rangle + \frac{1}{2} \int_0^{t} \langle l, \ufieldx (X_s^a) \rangle d\pcafx{s} - \int_0^{t} \langle l'', X_s^a \rangle ds \\ \forall \, t \geq 0 \, \text{$P_z^a$-a.s.}
    \end{multline}
    In this equation, $\pcafx{t}$ denotes the real valued PCAF associated with $\vmeasx$ by the Revuz correspondence.

    The analogous statement holds for $\dens$.
\end{corollary}
\begin{proof}
    Since it holds that $\densa$, $\dens \in \mathrm{QR}(\ltwo) \cap \mathrm{BV}(\ltwo, \ho)$, see Lemma \ref{lem:L1+_RayHamza}, Remark \ref{rem:RayHamza} and Theorem \ref{lem:gen_ibp}, this corollary follows from a direct application of \cite{RZZ12}, Theorem 3.2.
\end{proof}

In order to attain a deeper understanding of the process associated to $\mathcal{E}^{a}$, we now examine the support of $\vmeasx$. However, it is straightforward to prove that the support of $\vmeasx$ w.r.t.\ the topology induced by the $\ltwo$-norm is trivial, i.e.\ the entire space $\ltwo$. Therefore it is fruitful to restrict $\vmeasx$ to $\coec$ and compute the support w.r.t.\ the topology induced by the supremum norm, which is finer than that induced by the $\ltwo$-norm, allowing for a smaller and more insightful support.

\begin{lemma}[Support of $\lmeasx$ and $\vmeasx$]\label{lem:supp_lmeasx}
    The set $\setza_0 \coloneqq \{z \in \hc \mid z_a=0\}$ is a measurable set of full measure w.r.t.\ $\lmeasx$ and $\vmeasx$, (where we identify $\setza_0$ with its canonical embedding into $\coec$,) but it is not closed as a subset of $(\co, \| \cdot \|_{\infty})$. Identifying the considered measures with their restrictions to $\borel(\ltwo)\cap \coec$ or $\borel(\co,\|\cdot\|_{\infty})$ (see Remark \ref{rem:embeddings_and_measurability}), it holds that
    \begin{equation*}
        \mathrm{supp} \, \lmeasx = \mathrm{supp} \vmeasx = \{z \in \co \mid z_a=0\},
    \end{equation*}
    where the support is defined w.r.t.\ the topology on $\co$ induced by $\| \cdot \|_{\infty}$.
\end{lemma}
\begin{proof}
    Since $\vmeasx$ and $\lmeasx$ only differ by a multiplication of a non-negative constant, see \eqref{eqn:sigma_mu_connection}, it suffices to show the considered properties for $\lmeasx$.
    
    First, we set $\setza_{m,r} \coloneqq \{ z \in \hc_m \mid z_a \in \overline{B_r(0)} \}$ and $\setza_{m,0} \coloneqq \{ z \in \hc_m \mid z_a = 0\}$ for all $m \in \N \cup \{\infty\}$ and $r \in \R_{>0}$, where $\hc_{\infty} \coloneqq \hc$. Note that $\setza_0 = \setza_{\infty, 0}$. For $m < \infty$ and $r\in\R_{\geq 0}$, $\setza_{m,r}$ is a compact subset of $(\co, \|\cdot\|_{\infty})$, since by Lemma \ref{lem:hoelder_set_props}~(i), $\hc_m$ is compact, and $\setza_{m,r}$ is the intersection of $\hc_m$ with the closed set $\{z\in\co\mid z_a\in \overline{B_r(0)}\}$ for $r>0$, and with the closed set $\{z\in\co\mid z_a = 0\}$ for $r=0$. It follows that, for $m < \infty$ and $r\in\R_{\geq 0}$, $\setza_{m,r}$ is closed and Borel measurable in $(\co, \| \cdot \|_{\infty})$, and $\iota (\setza_{m,r})$ is closed and Borel measurable in $(\ltwo, \| \cdot \|_{\ltwo})$. Since $\setza_{m,0} \uparrow \setza_{\infty, 0}$ as $m \uparrow \infty$, it follows that $\setza_0=\setza_{\infty, 0}$ is Borel measurable.

    Now, choose for any $k \in \N$ and $r \in \R_{>0}$ an integer $m(k,2r,a)\in \N$ as in Lemma \ref{lem:roughness_of_pinning}, where we w.l.o.g.\ require that $m(\cdot, 2r,a)$ is strictly monotonically increasing. Then it holds
    \begin{align*}
    	&\approxmeasx(\setza_{m(k,2r,a), r}^c)
    	= \Emubb\left[ \ind_{\setza_{m(k,2r,a), r}^c}(\bb) \, \gmoll(\bb_a) \right]\\
    	&= \int_{\R} \Emubb\left[ \ind_{\setza_{m(k,2r,a), r}^c}(\bb^{\pin,a} + yl^a) \right] \gmoll(y) \, \mathcal{N}(0, a(1-a))(dy)\\
    	&\leq \int_{\overline{B_r(0)}} \Emubb\left[ \ind_{\hc_k^c}(\bb) \right] \gmoll(y) \, \mathcal{N}(0, a(1-a))(dy) + \int_{\overline{B_r(0)}^c} \gmoll(y) \,\mathcal{N}(0, a(1-a))(dy)\\
    	&\leq \frac{1}{\sqrt{2\pi a(1-a)}} \left( \mubb(\hc_k^c) + \int_{\overline{B_r(0)}^c} \gmoll(y) \, dy \right),
    \end{align*}
    where the second line is a consequence of Lemma \ref{lem:independence}, and the third line holds due to the following:
    If we have that $y \in \overline{B_r(0)}\,(\subset B_{2r}(0))$ and further $(\bb^{\pin,a} + yl^a) \in \setza_{m(k,2r,a), r}^c$, it has to hold that $(\bb^{\pin,a} + yl^a) \in \hc_{m(k,2r,a)}^c$, and hence $\bb \in \hc_k^c$ due to Lemma \ref{lem:roughness_of_pinning}.

    An application of the Portmanteau theorem for finite measures, using that $\setza_{m(k,2r,a),r}^c$ is open, yields for any $r>0$ that
    \begin{align*}
        \lmeasx (\setza_{m(k,2r,a), r}^c)
        &\leq \liminf_{n \to \infty} \mu^{(a,\varepsilon_{n,a})} (\setza_{m(k,2r,a), r}^c)\\
        &\leq \liminf_{n \to \infty} \frac{1}{\sqrt{2\pi a(1-a)}} \left( \mubb(\hc_k^c) + \int_{\overline{B_r(0)}^c} \gamma_{\varepsilon_{n,a}}(y) \,dy \right)\\
        &\leq \frac{1}{\sqrt{2\pi a(1-a)}} \left( \mubb(\hc_k^c) + \liminf_{n \to \infty}\int_{\overline{B_r(0)}^c} \gamma_{\varepsilon_{n,a}}(y) \,dy \right)\\
        &= \frac{1}{\sqrt{2\pi a(1-a)}} \, \mubb(\hc_k^c),
    \end{align*}
    where the last equality holds due to the fact that $(\gmoll)_{\varepsilon>0}$ is an approximate identity and $\varepsilon_{n,a} \to 0$ as $n\to\infty$.

    Now, since $\setza_{m(k,2r,a),r} \uparrow \setza_{\infty, r}$ as $k \uparrow \infty$ and hence $m(k,2r,a) \uparrow \infty$ by construction, we obtain $\setza_{m(k,2r,a),r}^c \downarrow \setza_{\infty, r}^c$ as $k \uparrow \infty$. Since $\lmeasx$ is finite, continuity of measures from above yields that
    \begin{align*}
        \lmeasx (\setza_{\infty, r}^c)
        &= \liminf_{k \to \infty} \lmeasx (\setza_{m(k,2r,a), r}^c)
        \leq \liminf_{k \to \infty} \frac{1}{\sqrt{2\pi a(1-a)}} \, \mubb(\hc_k^c)\\
        &= \frac{1}{\sqrt{2\pi a(1-a)}} \, \mubb(\hc^c)
        = 0
    \end{align*}
    for all $r>0$. Since $\setza_{\infty, r} \downarrow \setza_{\infty, 0} = \setza_0$ and hence $\setza_{\infty, r}^c \uparrow \setza_{\infty, 0}^c = \setza_0^c$ as $r \downarrow 0$, continuity of measures from below yields $\lmeasx (\setza_0^c) = \lim_{r \downarrow 0} \lmeasx (\setza_{\infty, r}^c) = 0$. This proves the first claim.

    Further, the set $\setza_0$ is not closed in $(\co, \| \cdot \|_{\infty})$, since it is indeed possible to find sequences in $\hc$ which converge uniformly to non-H\"{o}lder continuous functions in $\co$, and the restriction ``$z_a = 0$'' can not prohibit that, which is straightforward to show. Moreover, recall that $\{z \in \co \mid z_a = 0\}$ is closed and hence Borel measurable in $(\co, \| \cdot \|_{\infty})$. Since it is a superset of $\setza_0$, $\{z \in \co \mid z_a = 0\}$ has full measure, and since it is also closed, we can derive that $\mathrm{supp}(\lmeasx) \subseteq \{z \in \co \mid z_a = 0\}$. We even have equality, which we prove now: Let $z \in \co$ s.th.\ $z_a = 0$, $r>0$, and let $B_r^{\infty}(z)$ denote the $r$-ball w.r.t.\ $\| \cdot \|_{\infty}$ with center $z$. Then we can calculate for any $\varepsilon>0$ that
    \begin{align*}
    	\approxmeasx (\overline{B_{r/2}^{\infty}(z)})
    	&= \int_{\R} \E\left[\ind_{\overline{B_{r/2}^{\infty}(z)}}(\bb^{\pin,a}+yl^a)\right] \gmoll(y) \, \mathcal{N}(0,a(1-a))(dy)\\
    	&\geq \int_{B_{s_1}(0)} \E \left[\ind_{B_{s_2}^{\infty}(z)}(\bb)\right] \gmoll(y) \, \mathcal{N}(0,a(1-a))(dy)\\
    	&= \mubb(B_{s_2}^{\infty}(0)) \int_{B_{s_1}(0)} \gmoll(y) \, \mathcal{N}(0,a(1-a))(dy),
    \end{align*}
	where the first line is a consequence of Lemma \ref{lem:independence}, and the second line holds for $s_1\coloneqq r/6$ and $s_2\coloneqq r/3$. Indeed, if $y\in B_{s_1}(0)$ and $\bb\in B_{s_2}^{\infty}(z)$, it follows that $\bb^{\pin,a}+yl^a \in \overline{B_{r/2}^{\infty}(z)}$, since
	\begin{align*}
		&\| \bb^{\pin,a}+yl^a - z\|_{\infty}
		= \| \bb - z\|_{\infty} + |\bb_a-y| \, \|l^a\|_{\infty}\\
		&\leq \| \bb - z\|_{\infty} + (|\bb_a-z_a|+|y|)\cdot 1
		< s_1+2s_2=r/2.
	\end{align*}
	Subsequently, we can estimate
	\begin{align*}
		\lmeasx(B_r^{\infty}(z)) &\geq \lmeasx(\overline{B_{r/2}^{\infty}(0)})
		\geq \limsup_{n\to\infty} \mu_{(a,\varepsilon_{n,a})} (\overline{B_{r/2}^{\infty}(0)})\\
		&\geq \mubb(B_{s_2}^{\infty}(0)) \limsup_{n\to\infty} \int_{B_{s_1}(0)} \gamma_{\varepsilon_{n,a}}(y) \, \mathcal{N}(0,a(1-a))(dy)\\
		&= \frac{1}{\sqrt{2\pi a(1-a)}}\,\mubb(B_{s_2}^{\infty}(0)) > 0,
	\end{align*}
	where the first line follows due to monotonicity of measures and the aforementioned Portmanteau theorem for finite measures, and the final line holds due to properties of the approximate identity $(\gmoll)_{\varepsilon>0}$, and the fact that $\mubb$ has full topological support when viewed as a Borel measure on $(\co, \| \cdot \|_{\infty})$. This implies that $z\in\supp(\lmeasx)$, which finishes the proof.
\end{proof}

We define $\Prob_m \coloneqq \int_{\ltwo} \Prob_z \, dm(z)$ for every measure $m$ on $(\ltwo, \borel(\ltwo))$.

\begin{lemma}\label{lem:process_continuity}
	The process $X^a$ is almost surely continuous in the sense that
	\begin{equation*}
		\Prob_z[X_t^a \in \coec \text{ for a.e. } t\in [0,\infty)] = 1 \quad\text{for all } z\in\ltwo.
	\end{equation*}
	The same statement holds for $\mathcal{E}$-q.e.\ $z\in\ltwo$ for $X$.
\end{lemma}
\begin{proof}
	For this proof, we closely follow the strategy of the proof of \cite{RZZ12}, Theorem 6.5: Due to \cite{FOT11}, (5.1.13), the transfer method described in \cite{MR92}, Chapter 6 (from now on, the referred transfer method is applied every time when we cite or use concepts of \cite{FOT11}, and we don't point this out explicitly anymore), the fact that the PCAF $(t)_t$ and $\densa\mubb$ are in Revuz correspondence and the facts that $\mubb(\coec)=1$ and $p_s 1 = 1$ for all $s>0$, we obtain
	\begin{align*}
		\E_{1\cdot(\densa\mubb)} \left[ \int_0^t \ind_{\coec^c}(X^a_s) \, ds \right] &= \int_0^t \int_{\ltwo} p_s 1 \, d(\ind_{\coec^c}\cdot(\densa\mubb)) \, ds\\
		&= t \, \densa\mubb(\coec^c) = 0
	\end{align*}
	for all $t>0$. An application of the monotone convergence theorem (alongside utilizing Tonelli's theorem for showing measurability of the integrands) subsequently implies
	\begin{equation*}
		\E_{\densa\mubb} \left[ \int_0^{\infty} \ind_{\coec^c}(X^a_s) \, ds \right] = 0.
	\end{equation*}
	By virtue of the Markov property of $X^a$, one can argue that the function
	\begin{equation*}
		\ltwo \incl z \mapsto \E_{z} \left[ \int_0^{\infty} \ind_{\coec^c}(X^a_s) \, ds \right] \in \R
	\end{equation*}
	is $0$-excessive, and hence finely continuous, see \cite{FOT11}, Theorem A.2.7. Now, arguing with \cite{FOT11}, Lemma 4.1.5, we obtain
	\begin{equation*}
		\E_z \left[ \int_0^{\infty} \ind_{\coec^c}(X^a_s) \, ds \right] = 0, \text{ i.e.\ } \Prob_z \left\{\int_0^{\infty} \ind_{\coec^c}(X^a_s) \, ds = 0 \right\} = 1
	\end{equation*}
	for $\mathcal{E}^a$-q.e.\ $z\in\ltwo$. This result can be enhanced to every $z\in\ltwo$: Due to measurability of $X^a$ and Tonelli's theorem, the set
	\begin{equation*}
		\Lambda_0 \coloneqq \{X_t^a\in\coec \text{ for a.e. } t\in [0,\infty)\}
	\end{equation*}
	is measurable, and the previous equation implies that $\Prob_z(\Lambda_0)=1$ for $\mathcal{E}^a$-q.e.\ $z\in\ltwo$. Furthermore, the measure $\densa\mubb$ is log-concave, which follows from the log-concavity of the Gaussian measure $\mubb$ and the convexity of $\{z\in\coec \mid z_a\geq 0\}\in\borel(\ltwo)$. Thus, \cite{ASZ09}, Remark 7.6, yields that the semigroup associated with $X^a$ is strong Feller. Arguing as in the proof of \cite{DPR02}, Lemma 7.1, this yields
	\begin{equation*}
		\Prob_z\{\theta_t^{-1}\Lambda_0\} = 1
	\end{equation*}
	for all $z\in\ltwo$ and $t>0$. Since $\theta_{1/n}^{-1}\Lambda_0 \downarrow \Lambda_0$ as $n\to\infty$, continuity of measures from above yields $\Prob_z(\Lambda_0)=1$ for all $z\in\ltwo$, which proves the claim for $X^a$.
	
	With the same arguments, one can show that
	\begin{equation*}
		\Prob_z[X_t \in \coec \text{ for a.e. } t\in [0,\infty)] = 1
	\end{equation*}
	for $\mathcal{E}$-q.e.\ $z\in\ltwo$.
\end{proof}

\begin{remark}
	Since $\dens\mubb$ is not log-concave (contrary to $\densa\mubb$), which is straightforward to check, the strategy applied to prove Lemma \ref{lem:process_continuity} does not provide that the semigroup associated with $X$ is strong Feller, and therefore the extension argument to general $z\in\ltwo$ breaks down.
\end{remark}

\begin{lemma}
	The process $X^a$ is almost surely greater than or equal to zero at $a$, i.e.\ it holds
	\begin{equation*}
		\Prob_z[\mkern1.5mu \overline{\mkern-1.5mu X_t^a \mkern-1.5mu}\mkern1.5mu(a) \geq 0 \text{ for a.e. } t\in [0,\infty)] = 1 \quad \text{for all } z\in\ltwo.
	\end{equation*}
	Further, in the sense of random measures on $[0,\infty)$, it holds
	\begin{equation*}
		\ind_{ \{ z\in\coec \mid \bar{z}_a=0 \} }(X^a_s) \, d\pcafx{s} = d\pcafx{s} \quad \text{$P^a_z$-a.s.\ for $\mathcal{E}^a$-q.e.\ } z\in\ltwo. 
	\end{equation*}
\end{lemma}
\begin{proof}
	The both statements follow by an adaption of the proof of Lemma \ref{lem:process_continuity}: Considering that $\densa\mubb(\{z\in\coec\mid \bar{z}_a \geq 0\}^c) = 0$, the first statement is a straightforward enhancement of the aforementioned proof. Furthermore, since the PCAF $\pcafx{}$ is in Revuz correspondence with the variation measure $\vmeasx$ and $\supp \vmeasx = \{z\in\coec\mid\bar{z}_a\geq 0\} =: \setza_0^{\ast}$, see Lemma \ref{lem:supp_lmeasx}, we observe that
	\begin{equation*}
		\E_{1\cdot(\densa\mubb)} \left[ \int_0^t (1-\ind_{\setza_0^{\ast}}(X^a_s)) \, d\pcafx{s} \right]
		= t \, \vmeasx((\setza_0^{\ast})^c) = 0,
	\end{equation*}
	and, similarly as in the proof of Lemma \ref{lem:process_continuity}, we infer that
	\begin{align*}
		&\int_0^{\infty} (1-\ind_{\setza_0^{\ast}}(X^a_s)) \, d\pcafx{s} = 0 \quad \text{$P^a_z$-a.s.\ for $\mathcal{E}$-q.e.\ } z\in\ltwo \text{ and all } t>0\\
		\Rightarrow &\int_0^{t} \ind_{\setza_0^{\ast}}(X^a_s) \, d\pcafx{s} = \int_0^{t} 1 \, d\pcafx{s} \quad \text{$P^a_z$-a.s.\ for $\mathcal{E}^a$-q.e.\ } z\in\ltwo \text{ and all } t>0,
	\end{align*}
	which yields the claim.
\end{proof}
\begin{remark}[Weakly solved SPDEs]
	The following is a purely heuristic (!) remark which nonetheless provides some useful intuition about the behavior of the considered processes. Heuristically, the process $X^a$ weakly solves the SPDE
	\begin{equation*}
		dX_t^a = \frac{1}{2} \Delta X^a_t \, dt + dW_t + \delta_a \otimes dl_t^{0,X_{\cdot}^a(a)}.
	\end{equation*}
	Intuitively, the drift term consists of a singular, upwards ``Delta force'' at location $a$ which acts exactly when the process reaches zero at $a$, causing $\overline{X^a_{\cdot}}(a)$ to almost surely stay non-negative. Further, again heuristically, the process $X$ weakly solves the SPDE
	\begin{equation*}
		dX_t = \frac{1}{2} \Delta X_t \, dt + dW_t + dl_t^{0,X_{\cdot}},
	\end{equation*}
	where the drift term describes an ``upwards force'' at location $x\in(0,1)$ exactly when the process is zero at $x$.
	
	These highly ill-defined drift terms can be found by arguing as if the strictly distributional RHSs of the respective IbPFs could be written in terms of a scalar product, and then applying standard techniques for deriving the weakly solved SPDE from the respective IbPF. We want to stress that the last rigorous statement towards the SPDEs solved weakly by the considered processes are the Fukushima decompositions found in Corollary \ref{cor:ProcDecomp}.
\end{remark}

\begin{remark}[Generalizations, Discussion, Open Problems]
	Determining the support of $\vmeas$ is still an open problem. Moreover, the behavior of $\pcaf{}$ needs more thorough examination. Further, it is worth investigating whether $\dens \in \mathrm{BV}(\ltwo, \ltwo)$ holds, as there are valid indications that this might be the case, potentially providing even more regularity of $\dens$.
\end{remark}

\appendix

\section{Some Technical Results}\label{appendix}

\begin{lemma}\label{lem:DDapprox}
	(An Approximation for Donsker's Delta) For all $x\in (0,1)$ it holds that
	\begin{equation}\label{eqn:DDapprox_conv}
		L_+^2(\SchwartzDistribs, \mu) \incl \gmoll(\bb_x) \xrightarrow{\varepsilon \downarrow 0} \DDx \;\,\text{in $\HidaDistribs$},
	\end{equation}
	and for all $\varepsilon >0$, $x\in (0,1)$ and $\varphi \in \SchwartzFcts$ it holds
	\begin{align*}
		&\mathrm{T}(\gmoll(\langle q_x, \cdot \rangle))(\varphi)
		= \frac{1}{\sqrt{2\pi (x(1-x)+\varepsilon)}}\\
		&\quad\quad\quad \cdot\exp\left( -\frac{1}{2}\int_{\R}\varphi^2ds + \left(\frac{1}{2x(1-x)}-\frac{\varepsilon}{2(x(1-x)+\varepsilon)}\right) \left(\int_{\R} q_x \varphi ds\right)^2\right),
	\end{align*}
	and further, for all $z\in\C$, we have
	the bound
	\begin{equation}\label{eqn:ufunct_bounded}
		|\mathrm{T}(\gmoll(\langle q_x, \cdot \rangle))(z\varphi)|\leq \frac{1}{\sqrt{2\pi x(1-x)}}\exp\left(\frac{3}{2}|z|^2\|\varphi\|_{L^2(\R)}^2\right).
	\end{equation}
\end{lemma}
\begin{proof}
	The $\mathrm{T}$-transform of $\gmoll(\langle q_x, \cdot \rangle)$ follows from a straightforward modification of the standard calculation of the $\mathrm{T}$-transform of $\DDx$. Further, \eqref{eqn:ufunct_bounded} follows immediately, and \eqref{eqn:DDapprox_conv} follows by an application of Theorem \ref{thm:conv_in_hidaspace}.
\end{proof}

\begin{lemma}\label{lem:int_onefold_moll}
	For all $h\in \ho \cap H^2 = D(A)$, the Bochner integral $\int_{(0,1)} \bar{h}_x \gmoll(\ddext(\cdot))\, dx$ is a well-defined element of $L^2(\ltwo, \mubb)$, i.e.\ it holds
	\begin{equation}\label{lem:BochInt_integrand}
		\left[I: (0,1) \to L^2(\ltwo, \mubb): x \mapsto \bar{h}_x \gmoll(\ddext(\cdot)) \right] \in \mathcal{L}(\lambda\rvert_{(0,1)}, L^2(\ltwo, \mubb)).
	\end{equation}
\end{lemma}
\begin{proof}
	Due to the dominated convergence theorem, $I$ is continuous and bounded. Thus, \cite{HNVW16}, Theorem 1.1.6 (Pettis measurability theorem) and Proposition 1.2.2, yield the claim.
\end{proof}

\begin{corollary}\label{cor:BInt_in_L2}
	For all $h\in \ho\cap H^2$, the Bochner integral $\int_{(0,1)} \bar{h}_x \gmoll(\bb_x) \, dx$ is a well-defined element of $L^2(\SchwartzDistribs,\mu)$.
\end{corollary}
\begin{proof}
	Clearly, the operator $T: L^2(\ltwo, \mubb) \to L^2(\SchwartzDistribs,\mu): f \mapsto f\circ\bb$ is linear and isometric, thus $T(I) \in \mathcal{L}(\lambda\rvert_{(0,1)}, L^2(\SchwartzDistribs, \mu))$, where $I$ is defined as in \eqref{lem:BochInt_integrand}, see e.g.\ \cite{HNVW16}, p.\ 15.
\end{proof}

\begin{lemma}\label{lem:BInt_DDx_in_Sprime}
	For all $h\in \ho\cap H^2$, the Bochner integral $\int_{(0,1)} \bar{h}_x \DDx dx$ exists as a Hida distribution.
\end{lemma}
\begin{proof}
	The claim follows by a straightforward application of Theorem \ref{thm:bochner_int_in_hidaspace}.
\end{proof}

\begin{lemma}\label{lem:conv_int_onefold_moll}
	Let $h\in \ho\cap H^2$. Then it holds
	\begin{equation*}
		\int_{(0,1)} \bar{h}_x\gmoll(\bb_x) \, dx
		\xrightarrow{\varepsilon\downarrow 0}
		\int_{(0,1)} \bar{h}_x \DDx \, dx \;\,\text{in $\HidaDistribs$}.
	\end{equation*}
\end{lemma}
\begin{proof}
	The claim follows by a straightforward application of Theorem \ref{thm:conv_in_hidaspace}.
\end{proof}

\begin{lemma}\label{lem:vecmeas_variationformula}
	(Variation Formula and Variation of Weak Limits) Let X be a metric space with Borel $\sigma$-field $\borel(X)$ and let $B$ be a Banach space. Then the following holds:
	\begin{itemize}
		\item[(i)] For all $V \in \ca(\borel(X),B)$ it holds
		\begin{equation}\label{eqn:variation_formula}
			\var{V} = \sup\bigg\{ \sum_{i \in I} \Big\| \int_X f_i \, dV \Big\|_B \,\Big|\, (f_i)_{i\in I} \text{ is a partition of unity} \bigg\},
		\end{equation}
		where we call $(f_i)_{i \in I}$ a \emph{partition of unity} iff $(f_i)_{i \in I}$ is a family of continuous (bounded) functions in $[0,1]^X$ with a finite index set $I$ s.th.\ it holds $\sum_{i\in I} f_i = 1$.
		\item[(ii)] Let $(V_n)$ be a sequence of vector measures in $\ca(\borel(X),B)$, converging weakly to the vector measure $V \in \ca(\borel(X),B)$. Then it holds $\var{V} \leq \liminf_n \var{V_n}$. In particular, if $(V_n)$ is uniformly bounded w.r.t.\ $\var{\cdot}$, then $V$ is of bounded variation.
	\end{itemize}
\end{lemma}
\begin{proof}
	The statement (i) is trivial for $\svar{V}=0$, and the case $\svar{V}=\infty$ is excluded by \cite{DU77}, pp.\ 6--7 and Corollary 19, p.\ 9. Thus it remains to consider the case $\svar{V} \in (0,\infty)$.
	
	Let $(E_i)_{i\in I}$ be a finite measurable partition of $X$, and assume w.l.o.g.\ that $I\neq\emptyset$. Let $\varepsilon > 0$. Since $V$ is regular, see \cite{MS94}, Lemma 1.2, for all $i\in I$, there exists a closed set $F_i\subseteq X$ and an open set $G_i\subseteq X$ such that $F_i \subseteq E_i \subseteq G_i$ and $\svar{V}(G_i\setminus F_i) < \varepsilon/|I|^2$. We define $U \coloneqq \bigcup_i (G_i \setminus F_i)\supseteq (\bigcup_i F_i)^c$ and observe that $\svar{V}(U)<\varepsilon/|I|$ due to sub-additivity of $\svar{V}(\cdot)$. A straightforward application of Urysohn's lemma and a normalization argument yields the existence of a partition of unity $(f_i)$ s.th.\ $f_i\rvert_{F_i} = 1$ and $f_i\rvert_{F_j} = 0$ for all $i,\,j\in I$ with $i\neq j$. Hence we obtain
	\begin{equation*}
		\Big\| V(E_i) - \int_X f_i \, dV \Big\|_B \leq \Big\| \int_X \ind_{E_i} - f_i \, dV \Big\|_B \leq \svar{V}(\supp (\ind_{E_i} - f_i)) \leq \svar{V}(U) < \frac{\varepsilon}{|I|}
	\end{equation*}
	for all $i\in I$, which yields $\sum_i \|V(E_i)\|_B - \varepsilon \leq \sum_i \| \int_X f_i \, dV \|_B.$ Thus, for every choice of a finite measurable partition $(E_i)$ and $\varepsilon >0$, there exists a partition of unity $(f_i)$ s.th.\ the previous estimate holds. This implies that $\var{V}$ is bounded above by the RHS of \eqref{eqn:variation_formula}.
	
	To prove the reverse estimate, let $(f_i)_{i \in I}$ be a partition of unity, and assume w.l.o.g.\ that $I \neq \emptyset$. Let $\varepsilon >0$. By a standard approximation argument, there exists a finite measurable partition $(E_j)_{j\in J}$ of $X$ and a family of simple functions $(g_i)_{i\in I}$ s.th.\ $g_{i} = \sum_j a_{i,j}\ind_{E_j} \in [0,1]^X$ with $g_i \leq f_i$ and $\| f_i -g_i \|_{\infty} < \varepsilon/(|I|\,\svar{V})$ for all $i\in I$. Then it holds
	\begin{equation*}
		\sum_{i\in I} \Big\| \int_X f_i \, dV \Big\|_B
		\leq \sum_{i\in I} \Big( \sum_{j\in J} |a_{i,j}| \| V(E_j) \|_B \Big) + \| f_i-g_i \|_{\infty} \svar{V}\\
		\leq \var{V} + \varepsilon,
	\end{equation*}
	where the last estimate follows from $\sum_{i} |a_{i,j}| = \sum_i g_i(x) \leq \sum_i f_i(x) = 1$ for all $j\in J$ and $x\in E_j$. Thus $\sum_i \| \int_X f_i \, dV \|_B \leq \var{V}$ holds, implying that the RHS of \eqref{eqn:variation_formula} is bounded above by $\var{V}$. This yields (i).
	
	To prove (ii), let $(f_i)_{i \in I}$ be a partition of unity and let $(V_n)_{n\in\N}$ converge weakly to $V$. Lower semicontinuity of the norm $\|\cdot\|_B$ w.r.t.\ weak convergence and (i) imply
	\begin{equation*}
		\sum_{i\in I} \Big\| \int_X f_i \, dV \Big\|_B
		\leq \liminf_{n\to\infty} \Big(\sum_{i\in I} \Big\| \int_X f_i \, dV_n \Big\|_B\Big)\\
		\leq \liminf_{n\to\infty} \var{V_n}.
	\end{equation*}
	Taking the supremum over all partitions of unity and applying (i) yields $\var{V} \leq \liminf_n \var{V_n}$.
\end{proof}

\bibliography{references}

@article {RZZ12,
	AUTHOR = {R\"ockner, Michael and Zhu, Rong-Chan and Zhu, Xiang-Chan},
	TITLE = {The stochastic reflection problem on an infinite dimensional
	convex set and {BV} functions in a {G}elfand triple},
	JOURNAL = {Ann. Probab.},
	FJOURNAL = {The Annals of Probability},
	VOLUME = {40},
	YEAR = {2012},
	NUMBER = {4},
	PAGES = {1759--1794},
	ISSN = {0091-1798,2168-894X},
	MRCLASS = {60J45 (26A45 31C25 52A07 60G60 60H15)},
	MRNUMBER = {2978137},
	DOI = {10.1214/11-AOP661},
	URL = {https://doi.org/10.1214/11-AOP661},
}

@article {GV18,
	AUTHOR = {Grothaus, Martin and Vosshall, Robert},
	TITLE = {Integration by parts on the law of the modulus of the
	{B}rownian bridge},
	JOURNAL = {Stoch. Partial Differ. Equ. Anal. Comput.},
	FJOURNAL = {Stochastic Partial Differential Equations. Analysis and
	Computations},
	VOLUME = {6},
	YEAR = {2018},
	NUMBER = {3},
	PAGES = {335--363},
	ISSN = {2194-0401,2194-041X},
	MRCLASS = {60H07 (31C25 35R60 46F25 60H40)},
	MRNUMBER = {3844653},
	DOI = {10.1007/s40072-018-0110-4},
	URL = {https://doi.org/10.1007/s40072-018-0110-4},
}

@article {MS94,
	AUTHOR = {M\"arz, Michael and Shortt, R. M.},
	TITLE = {Weak convergence of vector measures},
	JOURNAL = {Publ. Math. Debrecen},
	FJOURNAL = {Publicationes Mathematicae Debrecen},
	VOLUME = {45},
	YEAR = {1994},
	NUMBER = {1-2},
	PAGES = {71--92},
	ISSN = {0033-3883,2064-2849},
	MRCLASS = {28B05 (46G10 60B10)},
	MRNUMBER = {1291803},
	MRREVIEWER = {Donald\ L.\ Cohn},
	DOI = {10.5486/pmd.1994.1396},
	URL = {https://doi.org/10.5486/pmd.1994.1396},
}

@book {Par67,
	AUTHOR = {Parthasarathy, K. R.},
	TITLE = {Probability measures on metric spaces},
	SERIES = {Probability and Mathematical Statistics},
	VOLUME = {No. 3},
	PUBLISHER = {Academic Press, Inc., New York-London},
	YEAR = {1967},
	PAGES = {xi+276},
	MRCLASS = {60.00},
	MRNUMBER = {226684},
	MRREVIEWER = {R.\ A.\ Gangolli},
}

@book {Kle08,
    AUTHOR = {Klenke, Achim},
     TITLE = {Probability theory},
    SERIES = {Universitext},
      NOTE = {A comprehensive course,
              Translated from the 2006 German original},
 PUBLISHER = {Springer-Verlag London, Ltd., London},
      YEAR = {2008},
     PAGES = {xii+616},
      ISBN = {978-1-84800-047-6},
   MRCLASS = {60-01 (28-01 60B10 60F05 60F15 60G42 60J10)},
  MRNUMBER = {2372119},
MRREVIEWER = {Sophie\ Lemaire},
       DOI = {10.1007/978-1-84800-048-3},
       URL = {https://doi.org/10.1007/978-1-84800-048-3},
}

@book {VTC87,
    AUTHOR = {Vakhania, N. N. and Tarieladze, V. I. and Chobanyan, S. A.},
     TITLE = {Probability distributions on {B}anach spaces},
    SERIES = {Mathematics and its Applications (Soviet Series)},
    VOLUME = {14},
      NOTE = {Translated from the Russian and with a preface by Wojbor A.
              Woyczynski},
 PUBLISHER = {D. Reidel Publishing Co., Dordrecht},
      YEAR = {1987},
     PAGES = {xxvi+482},
      ISBN = {90-277-2496-2},
   MRCLASS = {60B11 (28C20 46G12)},
  MRNUMBER = {1435288},
       DOI = {10.1007/978-94-009-3873-1},
       URL = {https://doi.org/10.1007/978-94-009-3873-1},
}

@book {DS88,
    AUTHOR = {Dunford, Nelson and Schwartz, Jacob T.},
     TITLE = {Linear operators. {P}art {I}},
    SERIES = {Wiley Classics Library},
      NOTE = {General theory,
              With the assistance of William G. Bade and Robert G. Bartle,
              Reprint of the 1958 original,
              A Wiley-Interscience Publication},
 PUBLISHER = {John Wiley \& Sons, Inc., New York},
      YEAR = {1988},
     PAGES = {xiv+858},
      ISBN = {0-471-60848-3},
   MRCLASS = {47-01 (46-01)},
  MRNUMBER = {1009162},
}

@book {DU77,
    AUTHOR = {Diestel, J. and Uhl, Jr., J. J.},
     TITLE = {Vector measures},
    SERIES = {Mathematical Surveys},
    VOLUME = {No. 15},
      NOTE = {With a foreword by B. J. Pettis},
 PUBLISHER = {American Mathematical Society, Providence, RI},
      YEAR = {1977},
     PAGES = {xiii+322},
   MRCLASS = {28A45 (46B05 46G10)},
  MRNUMBER = {453964},
MRREVIEWER = {Robert\ E.\ Huff},
}

@article {AR90,
    AUTHOR = {Albeverio, Sergio and R\"ockner, Michael},
     TITLE = {Classical {D}irichlet forms on topological vector
              spaces---closability and a {C}ameron-{M}artin formula},
   JOURNAL = {J. Funct. Anal.},
  FJOURNAL = {Journal of Functional Analysis},
    VOLUME = {88},
      YEAR = {1990},
    NUMBER = {2},
     PAGES = {395--436},
      ISSN = {0022-1236,1096-0783},
   MRCLASS = {47B99 (31C25 46G12)},
  MRNUMBER = {1038449},
MRREVIEWER = {T.\ Hida},
       DOI = {10.1016/0022-1236(90)90113-Y},
       URL = {https://doi.org/10.1016/0022-1236(90)90113-Y},
}

@article {AZ20,
    AUTHOR = {Elad Altman, Henri and Zambotti, Lorenzo},
     TITLE = {Bessel {SPDE}s and renormalised local times},
   JOURNAL = {Probab. Theory Related Fields},
  FJOURNAL = {Probability Theory and Related Fields},
    VOLUME = {176},
      YEAR = {2020},
    NUMBER = {3-4},
     PAGES = {757--807},
      ISSN = {0178-8051,1432-2064},
   MRCLASS = {60H15 (60J46 60J55)},
  MRNUMBER = {4087483},
       DOI = {10.1007/s00440-019-00926-0},
       URL = {https://doi.org/10.1007/s00440-019-00926-0},
}

@book {HKPS93,
    AUTHOR = {Hida, Takeyuki and Kuo, Hui-Hsiung and Potthoff, J\"urgen and
              Streit, Ludwig},
     TITLE = {White noise},
    SERIES = {Mathematics and its Applications},
    VOLUME = {253},
      NOTE = {An infinite-dimensional calculus},
 PUBLISHER = {Kluwer Academic Publishers Group, Dordrecht},
      YEAR = {1993},
     PAGES = {xiv+516},
      ISBN = {0-7923-2233-9},
   MRCLASS = {60G20 (46F25 46Gxx 47N30 60G15 60H05 60H07 81T08)},
  MRNUMBER = {1244577},
MRREVIEWER = {Mylan\ Redfern},
       DOI = {10.1007/978-94-017-3680-0},
       URL = {https://doi.org/10.1007/978-94-017-3680-0},
}

@book {Kal97,
    AUTHOR = {Kallenberg, Olav},
     TITLE = {Foundations of modern probability},
    SERIES = {Probability and its Applications (New York)},
 PUBLISHER = {Springer-Verlag, New York},
      YEAR = {1997},
     PAGES = {xii+523},
      ISBN = {0-387-94957-7},
   MRCLASS = {60-01},
  MRNUMBER = {1464694},
MRREVIEWER = {F.\ B.\ Knight},
}

@book {Kuo96,
    AUTHOR = {Kuo, Hui-Hsiung},
     TITLE = {White noise distribution theory},
    SERIES = {Probability and Stochastics Series},
 PUBLISHER = {CRC Press, Boca Raton, FL},
      YEAR = {1996},
     PAGES = {xii+378},
      ISBN = {0-8493-8077-4},
   MRCLASS = {60G20 (46F25 46G12 46N30 60H05 60H20)},
  MRNUMBER = {1387829},
MRREVIEWER = {Isamu\ D\^oku},
}

@book {Oba94,
    AUTHOR = {Obata, Nobuaki},
     TITLE = {White noise calculus and {F}ock space},
    SERIES = {Lecture Notes in Mathematics},
    VOLUME = {1577},
 PUBLISHER = {Springer-Verlag, Berlin},
      YEAR = {1994},
     PAGES = {x+183},
      ISBN = {3-540-57985-0},
   MRCLASS = {60G15 (46G12 46N30 47N30 60H07 60J65 81S25)},
  MRNUMBER = {1301775},
MRREVIEWER = {Yuh-Jia\ Lee},
       DOI = {10.1007/BFb0073952},
       URL = {https://doi.org/10.1007/BFb0073952},
}

@article {AMMP10,
    AUTHOR = {Ambrosio, Luigi and Miranda, Jr., Michele and Maniglia,
              Stefania and Pallara, Diego},
     TITLE = {B{V} functions in abstract {W}iener spaces},
   JOURNAL = {J. Funct. Anal.},
  FJOURNAL = {Journal of Functional Analysis},
    VOLUME = {258},
      YEAR = {2010},
    NUMBER = {3},
     PAGES = {785--813},
      ISSN = {0022-1236,1096-0783},
   MRCLASS = {28C20 (60B11)},
  MRNUMBER = {2558177},
MRREVIEWER = {O.\ Lipovan},
       DOI = {10.1016/j.jfa.2009.09.008},
       URL = {https://doi.org/10.1016/j.jfa.2009.09.008},
}

@book {MR92,
    AUTHOR = {Ma, Zhi Ming and R\"ockner, Michael},
     TITLE = {Introduction to the theory of (nonsymmetric) {D}irichlet
              forms},
    SERIES = {Universitext},
 PUBLISHER = {Springer-Verlag, Berlin},
      YEAR = {1992},
     PAGES = {vi+209},
      ISBN = {3-540-55848-9},
   MRCLASS = {60J40 (31C25 60J45)},
  MRNUMBER = {1214375},
MRREVIEWER = {Tu\ Sheng\ Zhang},
       DOI = {10.1007/978-3-642-77739-4},
       URL = {https://doi.org/10.1007/978-3-642-77739-4},
}

@book {HNVW16,
    AUTHOR = {Hyt\"onen, Tuomas and van Neerven, Jan and Veraar, Mark and
              Weis, Lutz},
     TITLE = {Analysis in {B}anach spaces. {V}ol. {I}. {M}artingales and
              {L}ittlewood-{P}aley theory},
    SERIES = {Ergebnisse der Mathematik und ihrer Grenzgebiete. 3. Folge. A
              Series of Modern Surveys in Mathematics [Results in
              Mathematics and Related Areas. 3rd Series. A Series of Modern
              Surveys in Mathematics]},
    VOLUME = {63},
 PUBLISHER = {Springer, Cham},
      YEAR = {2016},
     PAGES = {xvi+614},
      ISBN = {978-3-319-48519-5; 978-3-319-48520-1},
   MRCLASS = {46-02 (42B35 46E30)},
  MRNUMBER = {3617205},
MRREVIEWER = {Adam\ Os\polhk ekowski},
}

@book {AF03,
    AUTHOR = {Adams, Robert A. and Fournier, John J. F.},
     TITLE = {Sobolev spaces},
    SERIES = {Pure and Applied Mathematics (Amsterdam)},
    VOLUME = {140},
   EDITION = {Second},
 PUBLISHER = {Elsevier/Academic Press, Amsterdam},
      YEAR = {2003},
     PAGES = {xiv+305},
      ISBN = {0-12-044143-8},
   MRCLASS = {46E35 (46-01 46-02 46B70 46Exx)},
  MRNUMBER = {2424078},
}

@book {FOT11,
    AUTHOR = {Fukushima, Masatoshi and Oshima, Yoichi and Takeda, Masayoshi},
     TITLE = {Dirichlet forms and symmetric {M}arkov processes},
    SERIES = {De Gruyter Studies in Mathematics},
    VOLUME = {19},
   EDITION = {extended},
 PUBLISHER = {Walter de Gruyter \& Co., Berlin},
      YEAR = {2011},
     PAGES = {x+489},
      ISBN = {978-3-11-021808-4},
   MRCLASS = {60J25 (28A12 31C45 60F10 60J40 60J45 60J55)},
  MRNUMBER = {2778606},
}

@article {ASZ09,
    AUTHOR = {Ambrosio, Luigi and Savar\'e, Giuseppe and Zambotti, Lorenzo},
     TITLE = {Existence and stability for {F}okker-{P}lanck equations with
              log-concave reference measure},
   JOURNAL = {Probab. Theory Related Fields},
  FJOURNAL = {Probability Theory and Related Fields},
    VOLUME = {145},
      YEAR = {2009},
    NUMBER = {3-4},
     PAGES = {517--564},
      ISSN = {0178-8051,1432-2064},
   MRCLASS = {60J60 (49Q20 60G07 82C31)},
  MRNUMBER = {2529438},
       DOI = {10.1007/s00440-008-0177-3},
       URL = {https://doi.org/10.1007/s00440-008-0177-3},
}

@article {DPR02,
    AUTHOR = {Da Prato, Giuseppe and R\"ockner, Michael},
     TITLE = {Singular dissipative stochastic equations in {H}ilbert spaces},
   JOURNAL = {Probab. Theory Related Fields},
  FJOURNAL = {Probability Theory and Related Fields},
    VOLUME = {124},
      YEAR = {2002},
    NUMBER = {2},
     PAGES = {261--303},
      ISSN = {0178-8051,1432-2064},
   MRCLASS = {60H15 (47D07 47N30)},
  MRNUMBER = {1936019},
MRREVIEWER = {Sandra\ Cerrai},
       DOI = {10.1007/s004400200214},
       URL = {https://doi.org/10.1007/s004400200214},
}

@article {H14,
	AUTHOR = {Hairer, M.},
	TITLE = {A theory of regularity structures},
	JOURNAL = {Invent. Math.},
	FJOURNAL = {Inventiones Mathematicae},
	VOLUME = {198},
	YEAR = {2014},
	NUMBER = {2},
	PAGES = {269--504},
	ISSN = {0020-9910,1432-1297},
	MRCLASS = {60H15 (35R60 60H40 81S20 82C28)},
	MRNUMBER = {3274562},
	MRREVIEWER = {Dora\ Sele\v si},
	DOI = {10.1007/s00222-014-0505-4},
	URL = {https://doi.org/10.1007/s00222-014-0505-4},
}

@article {GIP15,
	AUTHOR = {Gubinelli, Massimiliano and Imkeller, Peter and Perkowski,
	Nicolas},
	TITLE = {Paracontrolled distributions and singular {PDE}s},
	JOURNAL = {Forum Math. Pi},
	FJOURNAL = {Forum of Mathematics. Pi},
	VOLUME = {3},
	YEAR = {2015},
	PAGES = {e6, 75},
	ISSN = {2050-5086},
	MRCLASS = {60H15 (35S50)},
	MRNUMBER = {3406823},
	DOI = {10.1017/fmp.2015.2},
	URL = {https://doi.org/10.1017/fmp.2015.2},
}

@article {NP92,
	AUTHOR = {Nualart, D. and Pardoux, \'E.},
	TITLE = {White noise driven quasilinear {SPDE}s with reflection},
	JOURNAL = {Probab. Theory Related Fields},
	FJOURNAL = {Probability Theory and Related Fields},
	VOLUME = {93},
	YEAR = {1992},
	NUMBER = {1},
	PAGES = {77--89},
	ISSN = {0178-8051,1432-2064},
	MRCLASS = {60H15 (35R60 49J40)},
	MRNUMBER = {1172940},
	MRREVIEWER = {Ralf\ Manthey},
	DOI = {10.1007/BF01195389},
	URL = {https://doi.org/10.1007/BF01195389},
}

@article {H13,
	AUTHOR = {Hairer, Martin},
	TITLE = {Solving the {KPZ} equation},
	JOURNAL = {Ann. of Math. (2)},
	FJOURNAL = {Annals of Mathematics. Second Series},
	VOLUME = {178},
	YEAR = {2013},
	NUMBER = {2},
	PAGES = {559--664},
	ISSN = {0003-486X,1939-8980},
	MRCLASS = {35K59 (35B10 35B65 35R60 60G22 60H15 60K35)},
	MRNUMBER = {3071506},
	MRREVIEWER = {Alp\ O.\ Eden},
	DOI = {10.4007/annals.2013.178.2.4},
	URL = {https://doi.org/10.4007/annals.2013.178.2.4},
}

@book {FH14,
	AUTHOR = {Friz, Peter K. and Hairer, Martin},
	TITLE = {A course on rough paths},
	SERIES = {Universitext},
	NOTE = {With an introduction to regularity structures},
	PUBLISHER = {Springer, Cham},
	YEAR = {2014},
	PAGES = {xiv+251},
	ISBN = {978-3-319-08331-5; 978-3-319-08332-2},
	MRCLASS = {60-02 (34F05 35R60 60H07 60H10 60H15)},
	MRNUMBER = {3289027},
	MRREVIEWER = {Fabrice\ Baudoin},
	DOI = {10.1007/978-3-319-08332-2},
	URL = {https://doi.org/10.1007/978-3-319-08332-2},
}

@article {Zam05,
	AUTHOR = {Zambotti, Lorenzo},
	TITLE = {Integration by parts on the law of the reflecting {B}rownian
	motion},
	JOURNAL = {J. Funct. Anal.},
	FJOURNAL = {Journal of Functional Analysis},
	VOLUME = {223},
	YEAR = {2005},
	NUMBER = {1},
	PAGES = {147--178},
	ISSN = {0022-1236,1096-0783},
	MRCLASS = {60H07 (60J65)},
	MRNUMBER = {2139884},
	MRREVIEWER = {Hac\`ene\ Boutabia},
	DOI = {10.1016/j.jfa.2004.08.001},
	URL = {https://doi.org/10.1016/j.jfa.2004.08.001},
}

@article {Zam02,
	AUTHOR = {Zambotti, Lorenzo},
	TITLE = {Integration by parts formulae on convex sets of paths and
	applications to {SPDE}s with reflection},
	JOURNAL = {Probab. Theory Related Fields},
	FJOURNAL = {Probability Theory and Related Fields},
	VOLUME = {123},
	YEAR = {2002},
	NUMBER = {4},
	PAGES = {579--600},
	ISSN = {0178-8051,1432-2064},
	MRCLASS = {60H07 (60H15 60J55)},
	MRNUMBER = {1921014},
	MRREVIEWER = {Catherine\ Donati-Martin},
	DOI = {10.1007/s004400200203},
	URL = {https://doi.org/10.1007/s004400200203},
}

@article {Zam03,
	AUTHOR = {Zambotti, Lorenzo},
	TITLE = {Integration by parts on {$\delta$}-{B}essel bridges,
	{$\delta>3$} and related {SPDE}s},
	JOURNAL = {Ann. Probab.},
	FJOURNAL = {The Annals of Probability},
	VOLUME = {31},
	YEAR = {2003},
	NUMBER = {1},
	PAGES = {323--348},
	ISSN = {0091-1798,2168-894X},
	MRCLASS = {60H15 (31C25 35R60 37L40 60H07)},
	MRNUMBER = {1959795},
	MRREVIEWER = {Jan\ I.\ Seidler},
	DOI = {10.1214/aop/1046294313},
	URL = {https://doi.org/10.1214/aop/1046294313},
}

\end{document}